%% file: associative.tex
\documentclass[11pt,twoside,a4paper]{article}
\usepackage{color,amscd,amsmath,amssymb,amsthm,latexsym,stmaryrd,authblk,mathabx,shuffle,hyperref,tikz,tkz-tab} 
\usepackage[latin1]{inputenc}
\usepackage[T1]{fontenc}   
\usepackage[english]{babel}
\providecommand{\keywords}[1]{\textbf{\textit{Keywords.}} #1}
\providecommand{\AMSclass}[1]{\textbf{\textit{AMS classification.}} #1}

\newcommand{\rond}[1]{*++[o][F-]{#1}}
\newcommand{\opas}{\mathbf{As}_\Omega^2}

\newcommand{\K}{\mathbb{K}}
\newcommand{\Z}{\mathbb{Z}}

\newcommand{\dimK}{\mathrm{dim}_\K}

\newcommand{\cat}{\mathrm{cat}}
\newcommand{\sch}{\mathrm{schr}}
\newcommand{\as}{\mathbf{As}}
\newcommand{\eas}{\mathbf{EAS}}
\newcommand{\bfP}{\mathbf{P}}
\newcommand{\id}{\mathrm{Id}}
\newcommand{\sym}{\mathrm{Sym}}

\newcommand{\bfF}{\mathbf{F}}

\input{xy}
\xyoption{all}

\input{arbresbinaires.tex}


\title{Generalized associative algebras}
\date{}
\author{Lo\"\i c Foissy}
\affil{\small{Univ. Littoral Côte d'Opale, UR 2597
LMPA, Laboratoire de Mathématiques Pures et Appliquées Joseph Liouville
F-62100 Calais, France}.\\ Email: \texttt{foissy@univ-littoral.fr}}

\theoremstyle{plain}
\newtheorem{theo}{Theorem}[section]
\newtheorem{lemma}[theo]{Lemma}
\newtheorem{cor}[theo]{Corollary}
\newtheorem{prop}[theo]{Proposition}
\newtheorem{defi}[theo]{Definition}

\theoremstyle{remark}
\newtheorem{remark}{Remark}[section]
\newtheorem{notation}{Notations}[section]
\newtheorem{example}{Example}[section]

\begin{document}

\maketitle

\begin{abstract}
We study diverse parametrized versions of the operad of associative algebra, where the parameter are taken
in an associative semigroup $\Omega$ (generalization of matching or family associative algebras) 
or in its cartesian square (two-parameters associative algebras). 
We give a description of the free algebras on these operads, study their formal series and prove that they are Koszul
when the set of parameters is finite. We also study operadic morphisms between the operad of classical associative algebras
and these objects, and links with other types of algebras (diassociative, dendriform, post-Lie$\ldots$).
\end{abstract}

\\

\AMSclass{16S10 18M60 20M75 16W99}

\tableofcontents

\section*{Introduction}

Recently, numerous parametrization of well-known operads were introduced. Choosing a set of parameters $\Omega$,
any product defining the considered operad is replaced by a bunch of products indexed by $\Omega$
and various relations are defined on them, mimicking the relations defining the initial operads.
One can for example require that any linear span of the parametrized products also satisfy the relations
of the initial operads: this is the \emph{matching} parametrization. For example, matching Rota-Baxter algebras,
associative, dendriform, prelie algebras are introduced in \cite{ZhangGaoGuo,FoissyPrelie}.
Another way is to use  one or more semigroup structures on $\Omega$: this it the \emph{family} parametrization.
In this spirit, family Rota-Baxter algebras,  dendriform, prelie algebras are introduced and studied in
\cite{AguiarDendriform,ZhangGao,ZhangGaoManchon,ManchonZhang}. 
A way to obtain both these parametrizations for dendriform algebras is introduced in \cite{Foissydendriforme},
with the help of a generalization of diassociative semigroups, called extended diassociative semigroups (briefly, EDS).
Finally, a two-parameters version is given for dendriform algebras and prelie algebras is described in \cite{FoissyManchonZhang}.\\

Our aim in this paper is the study of these parametrizations for the operad of associative algebras,
which surprisingly did not receive a lot of attention for now. 
We start with two-parameters associative algebras \cite{FoissyManchonZhang}.
If $(\Omega,\rightarrow)$ is a semigroup,  an $\Omega$-two-parameters associative algebra is given products
$*_{\alpha,\beta}$, with $\alpha,\beta \in \Omega$, satisfying the following axiom:
\begin{align*}
(x*_{\alpha,\beta} y)*_{\alpha \rightarrow \beta,\gamma} z&=x*_{\alpha,\beta \rightarrow \gamma}(y*_{\beta,\gamma} z).
\end{align*}
When $(\Omega,*)=(\Z/2\Z,\times)$, the two-parameters $\Omega$-associative algebras were described
in \cite{ChapotonHivertNovelli}, as an operad on bicolored trees. 
When $\Omega$ is finite, the associated operad $\opas$ is finitely generated and quadratic. We prove that it is Koszul
(Proposition \ref{prop1.3}), and describe its Poincaré-Hilbert formal series $P(X)$:
if $|\Omega|=\omega>1$, then
\begin{align*}
P(X)&=\frac{1-\omega X-\sqrt{1+2\omega(1-2\omega)X+\omega^2X^2}}{2\omega(\omega-1)}\\
&=X+\omega^2X^2+(2\omega-1)\omega^3X^3+(5\omega^2-5\omega+1)\omega^4X^4\\
&+(2\omega-1)(7\omega^2-7\omega+1)\omega^5X^5+(42\omega^4-84\omega^3+56\omega^2-14\omega+1)\omega^6X^6
+\ldots
\end{align*}
We deduce a formula for the dimension $p_n(\omega)$ of $\opas(n)$ with the help of Narayana numbers,
(Corollary \ref{cor1.5}), as well as properties of $p_n(\omega)$, seen as a polynomial in $\omega$
(Corollary \ref{cor1.4}). 
We also give a combinatorial description of Koszul dual of $\opas$ in terms of words (Proposition \ref{prop1.7})
and a description of free ${\opas}^!$-algebras. \\

In the second and third parts of this paper, we introduce and study $\Omega$-associative algebras.
Here, the set of parameters $\Omega$ is given two operations $\rightarrow$ and $\triangleright$.
An $\Omega$-associative algebra is given bilinear products $*_\alpha$, with $\alpha \in \Omega$,
with the following axioms:
\[x*_\alpha (y *_\beta z)=(x*_{\alpha \triangleright \beta} y)*_{\alpha \rightarrow \beta} z.\]
In order to have a suitable parametrised operad, we impose that free $\Omega$-associative algebras are of the form
\[T_A(V)=\bigoplus_{n=1}^\infty \K\Omega^{\otimes (n-1)}\otimes V^{\otimes n}.\]
Tensors of $\K\Omega^{\otimes (n-1)}\otimes V^{\otimes n}$ will be called $A$-typed words of length $n$
and will be denoted $\alpha_1\ldots \alpha_{n-1}v_1\ldots v_n$. 
We impose that the products $*_\alpha$ satisfy, among other conditions, that for any $v_1,v_2\in V$, 
\[v_1*_\alpha v_2=\alpha v_1 v_2.\]
We prove in Theorem \ref{theo3.2} that this holds if, and only if, the triple $(\Omega,\rightarrow,\triangleright)$
satisfies the following axioms:
\begin{align*}
\alpha \rightarrow (\beta \rightarrow \gamma)&=(\alpha \rightarrow \beta) \rightarrow \gamma,\\
 (\alpha \triangleright (\beta \rightarrow \gamma))\rightarrow (\beta \triangleright \gamma)
&=(\alpha \rightarrow \beta)\triangleright \gamma,\\
(\alpha \triangleright (\beta \rightarrow \gamma))\triangleright (\beta \triangleright \gamma)
&=\alpha \triangleright \beta.
\end{align*}
Such a triple $(\Omega,\rightarrow,\triangleright)$ will be called an extended associative semigroup (briefly, EAS). 
For example:
\begin{itemize}
\item If $\Omega$ is a set, its trivial EAS structure is given, for any $\alpha,\beta\in \Omega$,
\[\alpha \rightarrow\beta=\beta \triangleright \alpha=\beta.\]
In this case, the $\Omega$-family algebras are the matching associative algebras \cite{ZhangGaoGuo};
the particular case when $\Omega$ contains two elements appears also in \cite{Pirashvili}.
The underlying operads are also used in \cite{CombeGiraudo}.
\item If $(\Omega,\rightarrow)$ is a semigroup, one can make it an EAS with, 
for any $\alpha,\beta\in \Omega$,
\[\alpha \triangleright\beta= \alpha.\]
In this case, the $\Omega$-family algebras are the family associative algebras of \cite{ZhangGao}.
\item If $(\Omega,\star)$ is a group, one can make it an EAS with, for any $\alpha,\beta\in \Omega$,
\begin{align*}
\alpha \rightarrow \beta&=\beta,&\alpha \triangleright \beta=\alpha\star \beta^{\star -1}.
\end{align*}
\end{itemize}
We give more examples of EAS, including a classification of EAS of cardinality two, in the second section.
We in fact generalize these results in a linear setting: we first observe that if $(\Omega,\rightarrow,\triangleright)$
is a set with two operations, we consider the map:
\begin{align*}
\phi&:\left\{\begin{array}{rcl}
\Omega^2&\longrightarrow&\Omega^2\\
(\alpha,\beta)&\longrightarrow&(\alpha \rightarrow \beta,\alpha \triangleright \beta),
\end{array}\right.
\end{align*}
Then $(\Omega,\rightarrow,\triangleright)$ is an EAS if, and only if:
\begin{align*}
(\id \times \phi)\circ (\phi\times \id)\circ (\id \times \phi)&=(\phi \times \id)\circ (\id \times \tau)\circ (\phi\times \id),
\end{align*}
where $\tau:\Omega^2\longrightarrow\Omega^2$ is the usual flip:
\begin{align*}
\tau&:\left\{\begin{array}{rcl}
\Omega^2&\longrightarrow&\Omega^2\\
(\alpha,\beta)&\longrightarrow&(\beta,\alpha).
\end{array}\right.
\end{align*}
This can easily be generalized in the category of vector spaces: a linear extended associative semigroup (briefly, $\ell$EAS)
is a pair $(A,\Phi)$, where $\Phi:A\otimes A\longrightarrow A\otimes A$ is a linear map such that:
\begin{align*}
(\id \otimes \Phi)\circ (\Phi\otimes \id)\circ (\id \otimes \Phi)&=(\Phi \otimes \id)\circ (\id \otimes \tau)\circ (\Phi\otimes \id),
\end{align*}
where $\tau:A\otimes A\longrightarrow A\otimes A$ is the usual flip.
In particular, if $(\Omega,\rightarrow,\triangleright)$ is an EAS, then its algebra $\K\Omega$ is an $\ell$EAS.
We then introduce the notion of $\Phi$-associative algebra (Definition \ref{defialgebras})
and we describe  free $\Phi$-associative algebras $T_\Phi(V)$  in term of tensor algebras in Theorem \ref{theo3.2}.
In particular, as a vector space,
\[T_\Phi(V)=\bigoplus_{n=1}^\infty A^{\otimes(n-1)}\otimes V^{\otimes n}.\]
We prove in Proposition \ref{prop3.3} that if $V$ is a $\Phi$-associative algebra, 
then $V\otimes A$ is naturally an associative algebra;
if $\Phi$ is invertible, we prove conversely that any convenient associative product
 on $V\otimes A$ gives rise to a $\Phi$-associative algebra structure on $V$.
Following these results, we study the algebra structure of $T_\Phi(V)\otimes A$ and,
if $\Phi$ is invertible, we prove that  it is freely generated by $V\otimes A$
(Proposition \ref{prop3.4}). 

The description of free $\Phi$-algebras induce a combinatorial description of the operad $\as_\Phi$
of $\Phi$-associative algebras (Proposition \ref{prop3.5}). We prove that, when $A$ is finite-dimensional, 
that the operad $\as_\Phi$ is Koszul, and that its Koszul dual is the operad of $\as_{\Phi^*}$-algebras,
generalizing  a well-known result for the operad $\as$ of "classical" associative algebras (Proposition \ref{prop3.6}
and Theorem \ref{theo3.7}). We study operad morphisms between the operad of associative algebras
and $\as_\Phi$, which is related to eigenvectors of $\Phi$ (Proposition \ref{propassociative}).
We then give results on operadic maps between the operads $\as$ and  $\as_\Phi$, and between the operads 
$\opas$ and $\as_\Phi$ (Propositions \ref{prop3.14} and \ref{prop3.15}).
The paper ends with various links with other types of algebras, such that diassociative, post-Lie, dendriform, tridendriform
or duplicial algebras, and their Koszul duals. \\

\textbf{Acknowledgements}. 
The author acknowledges support from the grant ANR-20-CE40-0007
\emph{Combinatoire Algébrique, Résurgence, Probabilités Libres et Opérades}.\\

\begin{notation} 
Let $\K$ be a commutative field. Any vector space in this text will be taken over $\K$.
\end{notation}

\section{Two-parameters $\Omega$-associative algebras}

\begin{notation}
In all this section, $(\Omega,\rightarrow)$ is an associative semigroup.
\end{notation}

\subsection{Definition}

In the spirit of the notion of two-parameters dendriform or duplicial algebras of \cite{FoissyManchonZhang},
we now introduce the notion of two-parameters associative algebras, which can be found in \cite{AguiarDendriform}:

\begin{defi}
A two-parameters $\Omega$-associative algebra is a family $(V,(*_{\alpha,\beta})_{\alpha,\beta \in \Omega})$,
where $V$ is a vector space and,  for any $(\alpha,\beta)\in \Omega^2$, $*_{\alpha,\beta}:V\otimes V\longrightarrow V$
is a linear map such that:
\begin{align*}
&\forall \alpha,\beta,\gamma \in \Omega,\: \forall x,y,z\in V,&
(x*_{\alpha,\beta} y)*_{\alpha \rightarrow \beta,\gamma} z&=x*_{\alpha,\beta \rightarrow \gamma}(y*_{\beta,\gamma} z).
\end{align*}
\end{defi}

\begin{remark}
If $|\Omega|=1$, $\Omega$-associative algebras are associative algebras.
\end{remark}

Such a structure is related to $(\Omega,\rightarrow)$-graded associative products on $V\otimes \K\Omega$. 
For the sake of simplicity, we shall denote
the tensor product $x\otimes \alpha$, with $x\in V$ and $\alpha \in \Omega$, by $x\alpha$. 

\begin{prop} \label{prop1.2}
Let $V$ be a vector space, endowed with bilinear products $*_{\alpha,\beta}$ for any $(\alpha,\beta)\in \Omega^2$.
We define a product $*$ on $V\otimes \K\Omega$ by:
\begin{align*}
&\forall x,y\in V,\: \forall \alpha,\beta \in \Omega,&x\alpha *y\beta&=x*_{\alpha,\beta} y(\alpha \rightarrow \beta).
\end{align*}
Then $*$ is associative if, and only if, $(V,(*_{\alpha,\beta})_{\alpha,\beta \in \Omega})$ is a
two-parameters $\Omega$-associative algebra.
\end{prop}

\begin{proof}
For any $x,y,z\in V$, any $\alpha,\beta,\gamma \in \Omega$:
\begin{align*}
(x\alpha * y\beta)*z\gamma
&=(x*_{\alpha,\beta} y)*_{\alpha \rightarrow \beta,\gamma} z (\alpha \rightarrow\beta\rightarrow \gamma),\\
x\alpha *(y\beta*z\gamma)
&=x*_{\alpha,\beta \rightarrow \gamma}(y*_{\beta,\gamma} z)(\alpha \rightarrow\beta\rightarrow \gamma).
\end{align*}
The result is then immediate. \end{proof}

\subsection{The operad of two-parameters $\Omega$-associative algebras}

We refer to \cite{LodayVallette,Markl,MarklShnider,Mendez,Yau} for notations and usual results on operads.

\begin{notation}
We denote by $\opas$ the nonsymmetric operad of two-parameters $\Omega$-associative algebras.
It is generated by $*_{\alpha,\beta} \in \opas(2)$, with $\alpha,\beta \in \Omega$, and the relations
\begin{align*}
&\forall \alpha,\beta,\gamma \in \Omega,&
*_{\alpha \rightarrow\beta,\gamma}\circ_1 *_{\alpha,\beta}=*_{\alpha,\beta \rightarrow \gamma}\circ_2*_{\beta,\gamma}.
\end{align*}
\end{notation}

We assume in this section that $\Omega$ is finite, of cardinality denoted by $\omega$. Then the components of $\opas$
are finite-dimensional, and the following Proposition allows to inductively compute their dimension:

\begin{prop}\label{prop1.3}
The operad $\opas$ is Koszul.
For any $n\geqslant 1$, let us put $p_n=\dimK(\opas(n))$ and
\[P(X)=\sum_{n=1}^\infty p_n X^n \in \mathbb{Q}[[X]].\]
Then: 
\begin{align}
\label{EQ2} p_n&=\omega (\omega-1)\sum_{k=1}^{n-1}p_kp_{n-k}+\omega p_{n-1},
\end{align}
or equivalently, if $|\omega|\geqslant 2$:
\begin{align}
\label{EQ1} 
P(X)&=\frac{1-\omega X-\sqrt{1+2\omega(1-2\omega)X+\omega^2X^2}}{2\omega(\omega-1)}.
\end{align}
\end{prop}

\begin{proof}
we shall use the rewriting method of \cite{Dotsenko,LodayVallette}. 
We shall write elements of the free nonsymmetric operad generated by $\opas(2)$ as planar trees which vertices
are decorated by elements of $\Omega^2$. We will write indices on the vertices
on the trees and put the corresponding decorations between parentheses, 
and we delete the symbols $*$ in order to enlighten the notations. 
For example, the operadic tree
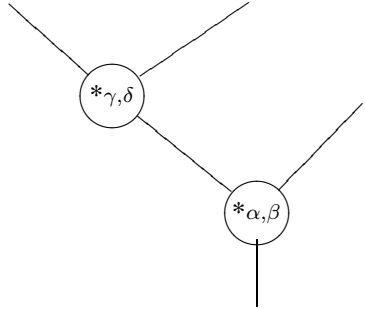
\[\xymatrix{\ar@{-}[rd]&&\ar@{-}[ld]&\\
&\rond{*_{\gamma,\delta}}\ar@{-}[rd]&&\ar@{-}[ld]\\
&&\rond{*_{\alpha,\beta}}\ar@{-}[d]&\\
&&&&}\]
will be shortly written $\bdtroisun((\alpha,\beta),(\gamma,\delta))$.
The rewriting rules are:
\[\bdtroisun((\alpha\rightarrow \beta,\gamma),(\alpha,\beta))\longrightarrow 
\bdtroisdeux((\alpha,\beta \rightarrow \gamma),(\beta,\gamma))\]
for any $\alpha,\beta,\gamma \in \Omega$. 
There is only one family of critical monomials, namely the monomials
\[\bdquatreun(((\alpha \rightarrow\beta)\rightarrow \gamma,
\delta), (\alpha \rightarrow\beta,\gamma), (\alpha,\beta)),\]
where $\alpha,\beta,\gamma \in \Omega$. 
Koszularity of $\opas$ comes from the confluence of the following diagram:
\[\xymatrix{&T_1\ar[rd] \ar[ld]&\\
T_2\ar[ddr]&&T_3 \ar[d]\\
&&T_4\ar[ld]\\
&T_5}\]
with:
\begin{align*}
T_1&=\bdquatreun(((\alpha \rightarrow\beta)\rightarrow \gamma,
\delta), (\alpha \rightarrow\beta,\gamma), (\alpha,\beta)),\\
T_2&=\bdquatrecinq((\alpha \rightarrow \beta,\gamma \rightarrow \delta),(\alpha,\beta),(\gamma,\delta)),\\
T_3&=\bdquatredeux((\alpha\rightarrow(\beta \rightarrow \gamma),\delta),(\alpha,\beta \rightarrow \gamma),
(\beta,\gamma)),\\
T_4&=\bdquatretrois((\alpha, (\beta \rightarrow \gamma)\rightarrow \delta), (\beta \rightarrow \gamma,\delta),
(\gamma,\delta)),\\
T_5&=\bdquatrequatre((\alpha,(\beta\rightarrow \gamma)\rightarrow \delta),(\beta,\gamma\rightarrow \delta),
(\gamma,\delta)),\\
&=\bdquatrequatre((\alpha,\beta\rightarrow (\gamma\rightarrow \delta)), (\beta,\gamma\rightarrow \delta),
(\gamma,\delta)).
\end{align*}
Hence, the operad $\opas$ is Koszul. Moreover, $\opas(n)$ has for basis the set of  rooted planar binary trees
with $n-1$ internal vertices decorated by $\Omega^2$, avoiding subtrees
\[\bdtroisun((\alpha\rightarrow \beta,\gamma),(\alpha,\beta)) \]
for any $\alpha,\beta,\gamma \in \Omega$.
For any planar binary tree $T$, let us denote by $p_T$ the number of decorations of the vertices of trees by elements of $\Omega^2$,
avoiding these subtrees. If $T$ is a planar binary tree different from the tree $\mid$ (which is the unit of the operad $\opas$), we denote
by $T_l$ the left subtree born from the root of $T$, by $T_r$ the right subtree born from the root of $T$,
and we write $T=T_l\vee T_r$. Then, looking at the possible decorations of the root:
\[p_T=p_{T_l}p_{T_r}\omega \times\begin{cases}
\omega \mbox{ if }T_2=\mid,\\
\omega-1 \mbox{ otherwise}.
\end{cases}\]
Hence, if $n\geqslant 2$, denoting by $\mathcal{T}_n$ the set of planar binary rooted trees with $n-1$ internal vertices:
\begin{align*}
\nonumber p_n&=\sum_{T\in \mathcal{T}_n} p_T\\
\nonumber &=\omega^2 \sum_{T\in \mathcal{T}_{n-1}}p_T+\omega(\omega-1)\sum_{k=2}^{n-1}
\sum_{T_l\in \mathcal{T}_k}\sum_{T_r\in \mathcal{T}_{n-k}}p_{T_l}p_{T_r}\\
\nonumber &=\omega^2 p_{n-1}+\omega(\omega-1)\sum_{k=2}^{n-1}p_kp_{n-k}\\
&=\omega (\omega-1)\sum_{k=1}^{n-1}p_kp_{n-k}+\omega p_{n-1},
\end{align*}
which gives (\ref{EQ2}). Summing over $n$, with $p_1=1$:
\begin{align}
\label{EQ3} P(X)&=\omega(\omega-1)P(X)^2+\omega XP(X)+X. 
\end{align}
if $\omega=1$, we obtain that
\[P(X)=XP(X)+X,\]
so
\[P(X)=\frac{X}{1+X}=\sum_{n=1}^\infty X^n,\]
recovering the formal series of the nonsymmetric operad of associative algebras. If $\omega\geqslant 2$,
solving (\ref{EQ3}), with the initial condition $P(0)=0$, we obtain (\ref{EQ1}). \end{proof}

\begin{example}
We obtain:
\begin{align*}
p_2(\omega)&=\omega^2,\\
p_3(\omega)&=(2\omega-1)\omega^3,\\
p_4(\omega)&=(5\omega^2-5\omega+1)\omega^4,\\
p_5(\omega)&=(2\omega-1)(7\omega^2-7\omega+1)\omega^5,\\
p_6(\omega)&=(42\omega^4-84\omega^3+56\omega^2-14\omega+1)\omega^6,\\
p_7(\omega)&=(2\omega-1)(66\omega^4-132\omega^3+84\omega^2-18\omega+1)\omega^7,\\
p_8(\omega)&=(429\omega^6-1287\omega^5+1485 \omega^4-825 \omega^3+225\omega^2-27\omega+1)\omega^8,\\
p_9(\omega)&=(2\omega-1)(715\omega^6-2145\omega^5+2431\omega^4-1287\omega^3+319\omega^2-33\omega+1)\omega^9.
\end{align*}
This gives:
\[\begin{array}{|c||c|c|c|c|c|c|c|c|c|}
\hline \omega\setminus n&1&2&3&4&5&6&7\\
\hline \hline 1&1&1&1&1&1&1&1\\
\hline 2&1&4&24&176&1440&12608&115584\\
\hline 3&1&9&135&2511&52245&1164213&27173475\\
\hline 4&1&16&448&15616&609280&25464832&1114882048\\
\hline 5&1&25&1125&63125&3965625&266890625&18816328125\\
\hline 6&1&36&2376&195696&18048096&1783238976&184576081536\\
\hline 7&1&49&4459&506611&64454845&8785674373&1254546699679\\
\hline 8&1&64&7680&1150976&193167360&34733293568&6542642380800\\
\hline 9&1&81&12393&2368521&506935665&116245810017&27925350157593\\
\hline \end{array}\]
\end{example}

\begin{remark}\begin{enumerate}
\item If $\omega=2$, the sequence $(p_n)_{n\geqslant 2}$ is referenced as \href{http://oeis.org/A156017}{A156017} 
in the \href{http://oeis.org/}{OEIS}. 
This is the sequence of dimensions of an operad given in \cite{ChapotonHivertNovelli}, 
generated by four products $\prec$, $\succ$, $\circ$  and $\odot$, with eight relations, see (21) in \cite{ChapotonHivertNovelli}. 
This is a special example of  a type of two-parameters $\Omega$-associative algebras, with $\Omega=(\Z/2\Z,\times)$ and:
\begin{align*}
*_{(\overline{0},\overline{0})}&=\circ,&
*_{(\overline{0},\overline{1})}&=\prec,&
*_{(\overline{1},\overline{0})}&=\succ,&
*_{(\overline{1},\overline{1})}&=\odot.
\end{align*}
Moreover:
\[P(X)_{\mid \omega=2}=\frac{1-2X-\sqrt{1-6(2X)+(2X)^2}}{4},\]
so for any $n\geqslant 2$, $p_n=2^{n-1} \sch_n$, where $\sch_n$ is the $n$-th large Schröder number (sequence 
\href{https://oeis.org/A006318}{A006318} in the \href{http://oeis.org/}{OEIS}):
\[\begin{array}{|c||c|c|c|c|c|c|c|c|c|c|}
\hline n&1&2&3&4&5&6&7&8&9&10\\
\hline \hline \sch_n&1&2&6&22&90&394&1806&8558&41586&206098\\
\hline\end{array}\]
\end{enumerate}
\end{remark}

\begin{cor} \label{cor1.4}
Let $n \geqslant 1$.
\begin{enumerate}
\item $p_n$ is a polynomial in $\Z[\omega]$, of degree $2n-2$. Its leading term is the $n$-th Catalan number $\cat_n$
(Sequence \href{https://oeis.org/A000108}{A000108} in the \href{http://oeis.org/}{OEIS}):
\[\begin{array}{|c||c|c|c|c|c|c|c|c|c|c|}
\hline n&1&2&3&4&5&6&7&8&9&10\\
\hline\hline \cat_n&1&1&2&5&14&42&132&429&1430&4862\\
\hline\end{array}\]
\item If $n\geqslant 2$, there exists a polynomial $q_n \in \Z[\omega]$, such that $p_n=\omega^n q_n$.
Moreover, $q_n(0)=(-1)^n$.
\item If $n$ is odd and $\geqslant 3$, then $p_n\left(\dfrac{1}{2}\right)=0$. 
\end{enumerate}
\end{cor}

\begin{proof}
1. and 2. We proceed by induction on $n$. As $p_1=1$, this is obvious. 
Let us assume the result at all ranks $<n$, with $n\geqslant 2$. The induction hypothesis gives that the following is a polynomial
in $\Z[\omega]$, of degree $2n-2$:
\[\omega (\omega-1)\sum_{k=1}^{n-1}p_kp_{n-k}.\]
Its leading term is 
\[\sum_{k=1}^{n-1}\cat_k \cat_{n-k}=\cat_n.\]
We also obtain that $\omega p_{n-1}$ is a polynomial in $\Z[\omega]$, of degree $2n-3$. Summing in (\ref{EQ2}),
we obtain the first point for $p_n$. Still by (\ref{EQ2}):
\begin{align*}
p_n&=\omega(\omega-1)\sum_{k=2}^{n-2}p_kp_{n-k}+2\omega(\omega-1)p_{n-1}+\omega p_{n-1}\\
&=\omega(\omega-1)\sum_{k=2}^{n-2}p_kp_{n-k}+\omega(2\omega-1)p_{n-1}\\
&=\omega^{n+1}(\omega -1)\sum_{k=2}^{n-2} q_kq_{n-k}+\omega^n(2\omega-1)q_{n-1}\\
&=\omega^n\left(\underbrace{\omega(\omega -1)\sum_{k=2}^{n-2} q_kq_{n-k}+(2\omega-1)q_{n-1}}_{q_n}\right).
\end{align*}
Moreover, $q_n(0)=0-q_{n-1}(0)=(-1)^n$, which proves the second point. \\

3. For $\omega=\dfrac{1}{2}$, we obtain:
\[P(X)_{\mid \omega=\frac{1}{2}}=X-2+\sqrt{4+X^2}=X+2\sum_{k=2}^\infty \frac{(-1)^k (2k-2)!}{2^{4k-1}k!(k-1)!}X^{2k}.
\qedhere\]
\end{proof}

\begin{cor} \label{cor1.5}
For any $n\geqslant 2$,
\[p_n=\frac{\omega^n}{n-1}\left(\sum_{k=1}^{n-1}\binom{n-1}{k}\binom{n-1}{k-1}\omega^{n-1-k}
(\omega-1)^{k-1}\right).\]
\end{cor}

\begin{proof}
Let us consider the Narayana numbers \cite{Narayana}:
\begin{align*}
&\forall k,n\geqslant 1,&N(n,k)&=\frac{1}{n}\binom{n}{k}\binom{n}{k-1},
\end{align*}
and their formal series
\[\mathbf{N}(z,t)=\sum_{k,n\geqslant 1}N(n,k) z^n t^{k-1}=\frac{1-z(t+1)-\sqrt{1-2z(t+1)+z^2(t-1)^2}}{2tz}.\]
Then:
\begin{align*}
P(X)&=X+X\mathbf{N}\left(\omega^2X,\frac{\omega-1}{\omega}\right)\\
&=X+\sum_{k,n\geqslant 1}N(n,k)\omega^{2n}X^{n+1}\left(\frac{\omega-1}{\omega}\right)^{k-1}\\
&=X+\sum_{n=2}^\infty \left(\sum_{k=1}^{n-1} N(n-1,k)\omega^{2n-2}\left(\frac{\omega-1}{\omega}\right)^{k-1}\right) X^n\\
&=X+\sum_{n=2}^\infty \left(\sum_{k=1}^{n-1} N(n-1,k)\omega^{2n-1-k}(\omega-1)^{k-1}\right) X^n.\qedhere
\end{align*}
\end{proof}

\begin{remark}
These numbers appear in \cite{ChenPan}, where they are interpreted in terms of Catalan paths. 
More precisely, with  the notations of \cite[Definition (1.13)]{ChenPan}:
\[p_n=\frac{1}{n-1}\sum_{k=1}^{n-1}\binom{n-1}{k}\binom{n-1}{k-1}(\omega^2)^{n-k}(\omega(\omega-1))^{k-1}
=C_{n-1}^{(\omega^2,\:\omega(\omega-1))}.\]
\end{remark}

\subsection{Koszul dual of $\opas$}

In all this paragraph, $(\Omega,\rightarrow)$ is a finite semigroup.

\begin{prop}
Koszul dual ${\opas}^!$ of the operad $\opas$ is the quotient of $\opas$ by the trees
\begin{align*}
&\bdtroisun((\alpha,\beta),(\gamma,\delta))=(\alpha,\beta)\circ_1(\gamma,\delta)
 \mbox{ with }\alpha\rightarrow \beta\neq \delta,\\
&\bdtroisdeux((\alpha,\beta),(\gamma,\delta))
=(\alpha,\beta)\circ_2(\gamma,\delta) \mbox{ with }\beta\neq \gamma\rightarrow \delta.
\end{align*}
In other terms, a ${\opas}^!$-algebra is a family $(V,(*_{\alpha,\beta})_{\alpha,\beta \in \Omega})$,
where $V$ is a vector space and,  for any $(\alpha,\beta)\in \Omega^2$, $*_{\alpha,\beta}:V\otimes V\longrightarrow V$
is a linear map such that:
\begin{align*}
&\forall \alpha,\beta,\gamma \in \Omega,\: \forall x,y,z\in V,&
(x*_{\alpha,\beta} y)*_{\alpha \rightarrow \beta,\gamma} z&=x*_{\alpha,\beta \rightarrow \gamma}(y*_{\beta,\gamma} z),\\
&\forall \alpha,\beta,\gamma,\delta \in \Omega,\:\forall x,y,z\in V,&
(x*_{\alpha,\beta} y)*_{\gamma,\delta} z&=0 \mbox{ if }\alpha \rightarrow \beta\neq \gamma,\\
&&x*_{\alpha,\beta}(y*_{\gamma,\delta} z)&=0\mbox{ if }\beta \neq \gamma \rightarrow \delta.
\end{align*}
For any $n\geqslant 2$, $\dimK({\opas}^!(n))=\omega^n$.

\end{prop}

\begin{proof}
The presentation of ${\opas}^!$ comes from a direct computation. 
Let us denote by $Q(X)$ the Poincaré-Hilbert formal series of ${\opas}^!$. As $\opas$ is Koszul, 
$Q(X)$ is the inverse for the composition of $-P(-X)$. From (\ref{EQ3}):
\begin{align*}
X&=\frac{P(X)-\omega(\omega-1)P(X)^2}{\omega P(X)+1}
=\frac{-P(-X)+\omega(\omega-1)(-P(-X))^2}{1-\omega(-P(-X))},
\end{align*}
so $Q(X)$ is given by:
\begin{align*}
Q(X)&=\frac{\omega(\omega-1)X^2+X}{1-\omega X}=X+\sum_{n=2}^\infty \omega^n X^n. \qedhere
\end{align*} \end{proof}

Let us give a combinatorial presentation of ${\opas}^!$:

\begin{prop}\label{prop1.7}
For any $n\geqslant 1$, let us put $\bfP(n)=(\K \Omega)^{\otimes n}$. Elements of $\bfP(n)$ are linear spans of words
$\alpha_1\ldots \alpha_n$ in $\Omega$. Then $\bfP=(\bfP(n))_{n\geqslant 1}$ is given a structure of nonsymmetric operad
with the following composition: for any $\alpha_1,\ldots,\alpha_n \in \Omega$,
for any $\beta_{i,j}\in \Omega$,
\[\alpha_1\ldots \alpha_n \circ (\beta_{1,1}\ldots \beta_{1,k_1},\ldots,\beta_{n,1}\ldots \beta_{n,k_n})
=\left(\prod_{i=1}^n \delta_{\alpha_i, \beta_{i,1}\rightarrow \ldots\rightarrow \beta_{i,k_i}}\right)
\beta_{1,1}\ldots \beta_{1,k_1}\ldots \beta_{n,1}\ldots \beta_{n,k_n}.\]
The unit is:
\[I=\sum_{\alpha \in \Omega} \alpha.\]
We define a suboperad $\bfP_0$ isomorphic to ${\opas}^!$ by:
\[\bfP_0(n)=\begin{cases}
\K I\mbox{ if }n=1,\\
\bfP(n)\mbox{ if }n\geqslant 2.
\end{cases}\]
\end{prop}

\begin{proof}
For any word  $w=\alpha_1\ldots\alpha_n$ in $\alpha$, we put $|w|=\alpha_1\rightarrow \ldots \rightarrow \alpha_n$.
The composition $\circ$ can be rewritten in the following way: for any words $w,w_1,\ldots,w_n$ with letters in $\Omega$,
$w$ being of length $n$,
\[w\circ (w_1,\ldots,w_n)=\delta_{w,|w_1|\ldots |w_n|} w_1\ldots w_n.\]
Let us prove the associativity: for any word $w$ of length $n$, $w_i$, $1\leqslant i\leqslant n$ of respective lengths $k_i$,
$w_{i,j}$ with $1\leqslant i \leqslant n$ and $1\leqslant j\leqslant k_i$, all with letters in $\Omega$:
\begin{align*}
&w\circ (w_1\circ (w_{1,1},\ldots,w_{1,k_1}),\ldots,w_n\circ (w_{n,1},\ldots,w_{n,k_n}))\\
&=\left(\prod_{i=1}^n \delta_{w_i,|w_{i,1}|\ldots |w_{i,k_i}|}
\delta_{w,\left| |w_{1,1}|\ldots |w_{1,k_1}|\right|\ldots \left| |w_{n,1}|\ldots |w_{n,k_n}|\right|}\right)w_{1,1}\ldots w_{n,k_n}\\
&=\left(\delta_{w_1\ldots w_n, |w_{1,1}|\ldots |w_{n,k_n}|}\delta_{w,|w_1|\ldots |w_n|}\right)w_{1,1}\ldots w_{n,k_n}\\
&=(w\circ (w_1,\ldots,w_n))\circ (w_{1,1},\ldots,w_{n,k_n}).
\end{align*}
Let us prove that $I$ is a unit. Let $w=\alpha_1\ldots \alpha_n$ be a word with letters in $\Omega$.
\begin{align*}
I\circ \alpha_1\ldots \alpha_n&=\sum_{\alpha\in \Omega} \delta_{\alpha,|\alpha_1\ldots \alpha_n}\alpha_1\ldots \alpha_n
=\alpha_1\ldots \alpha_n,\\
\alpha_1\ldots \alpha_n\circ (I,\ldots,I)&=\sum_{\beta_1,\ldots,\beta_n \in \Omega}
\prod_{i=1}^n \delta_{\alpha_i,\beta_i} \beta_1\ldots\beta_n\\
&=\alpha_1\ldots \alpha_n.
\end{align*}
So $\bfP$ is an operad. Obviously, $\bfP_0$ is a suboperad. Moreover, for any word $\alpha_1\ldots \alpha_n$
of length $\geqslant 3$:
\[\alpha_1\ldots \alpha_n=(\alpha_1\rightarrow \alpha_2)\alpha_3\ldots \alpha_n\circ (\alpha_1\alpha_2,I,\ldots,I).\]
A direct induction then proves that $\bfP_0$ is generated by $\bfP(2)$. Moreover, for any $\alpha,\beta,\gamma,\delta\in \Omega$:
\begin{align*}
(\alpha \rightarrow \beta)\gamma \circ (\alpha\beta,I)&=\alpha (\beta \rightarrow \gamma)\circ (I,\beta\gamma)
=\alpha\beta \gamma,\\
\gamma \delta\circ (\alpha\beta,I)&=0\mbox{ if }\alpha\rightarrow \beta \neq \gamma,\\
\alpha\beta \circ (I,\gamma\delta)&=0\mbox{ if }\gamma \rightarrow \delta \neq \beta.
\end{align*}
Therefore, the relations defining ${\opas}^!$ are satisfied in $\bfP_0$. 
Hence, there exists a surjective operad morphism:
\[\left\{\begin{array}{rcl}
{\opas}^!&\longrightarrow&\bfP_0\\
*_{\alpha,\beta}&\longrightarrow&\alpha\beta.
\end{array}\right.\]
Comparing the formal series of ${\opas}^!$ and $\bfP_0$, we deduce that this is an isomorphism. \end{proof}

\begin{cor}
Let $V$ be a vector space. The free ${\opas}^!$-algebra generated by $V$ is:
\[T_\Omega^2(V)=V\oplus \bigoplus_{n=2}^\infty (\K\Omega \otimes V)^{\otimes n}.\]
The product $*_{\alpha,\beta}$ is given in the following way: for any $u,v\in V$,
for any $\alpha_1u_1,\ldots, \alpha_ku_k$, 
$\beta_1v_1,\ldots, \beta_lv_l\in \K\Omega \otimes V$, with $k,l\geqslant 1$,
\begin{align*}
u*_{\alpha,\beta} v&=\alpha u\beta v,\\
\alpha_1u_1\ldots \alpha_ku_k *_{\alpha,\beta} v&=\left(\delta_{\alpha,\alpha_1\rightarrow\ldots \rightarrow \alpha_k}\right)
\alpha_1u_1\ldots \alpha_k u_k \beta v,\\
u *_{\alpha,\beta} \beta_1v_1\ldots \beta_lv_l&=\left(\delta_{\beta,\beta_1\rightarrow\ldots \rightarrow \beta_l}\right)
\alpha u \beta_1v_1\ldots \beta_lv_l,\\
\alpha_1u_1\ldots \alpha_ku_k*_{\alpha,\beta} \beta_1v_1\ldots \beta_lv_l&=
\left(\delta_{\alpha,\alpha_1\rightarrow\ldots \rightarrow \alpha_k}\delta_{\beta,\beta_1\rightarrow\ldots \rightarrow \beta_l}
\right)\alpha_1u_1\ldots \alpha_ku_k \beta_1v_1\ldots \beta_lv_l.
\end{align*}
\end{cor}

\section{Extended associative semigroups}

We here give the definition and few examples of extended associative semigroups.
More results can be found in \cite{FoissyEAS}.

\subsection{Definition and examples}

\begin{defi}
An associative extended semisgroup (briefly, EAS) is a triple $(\Omega,\rightarrow,\triangleright)$, where $\Omega$ is a
nonempty  set and $\rightarrow,\triangleright:\Omega^2\longrightarrow \Omega$ are maps such that, 
for any $\alpha,\beta,\gamma \in \Omega$:
\begin{align}
\label{EQ5} \alpha \rightarrow (\beta \rightarrow \gamma)&=(\alpha \rightarrow \beta) \rightarrow \gamma,\\
\label{EQ6} (\alpha \triangleright (\beta \rightarrow \gamma))\rightarrow (\beta \triangleright \gamma)
&=(\alpha \rightarrow \beta)\triangleright \gamma,\\
\label{EQ7} (\alpha \triangleright (\beta \rightarrow \gamma))\triangleright (\beta \triangleright \gamma)
&=\alpha \triangleright \beta.
\end{align}
\end{defi}

\begin{example}
\begin{enumerate}
\item Let $\Omega$ be a set.  We put:
\begin{align*}
&\forall \alpha,\beta \in \Omega,&\begin{cases}
 \alpha \rightarrow \beta=\beta,\\
\alpha \triangleright \beta=\alpha.
\end{cases} \end{align*}
Then $(\Omega,\rightarrow,\triangleright)$ is an EAS, denoted by $\eas(\Omega)$. 

\item Let $(\Omega,\star)$  be an associative semigroup. We put:
\begin{align*}
&\forall \alpha,\beta \in \Omega,& \alpha \triangleright \beta&=\alpha.
\end{align*}
Then $(\Omega,\star,\triangleright)$  is an EAS, which we denote by $\eas(\Omega,\star)$.  
\item Let $(\Omega,\star)$ be a group. We put, for any $\alpha,\beta\in \Omega$:
\begin{align*}
\alpha \rightarrow \beta&=\beta,&
\alpha \triangleright \beta&=\alpha \star \beta^{\star-1}.
\end{align*}
Then $(\Omega,\rightarrow,\triangleright)$ is an EAS, denoted by $\eas'(\Omega,\star)$. 
\end{enumerate}
\end{example}

\begin{defi} \label{defi9}
Let $(\Omega,\rightarrow,\triangleright)$ be an EAS. We shall say that it is nondegenerate if the following map is bijective:
\[\phi:\left\{\begin{array}{rcl}
\Omega^2&\longrightarrow&\Omega^2\\
(\alpha,\beta)&\longrightarrow&(\alpha\rightarrow\beta,\alpha \triangleright\beta).
\end{array}\right.\]
\end{defi}

\begin{example}\begin{enumerate}
\item Let $\Omega$ be a set. In $\eas(\Omega)$,  for any $\alpha,\beta \in \Omega$,
$\phi(\alpha,\beta)=(\beta,\alpha)$, so $\eas(\Omega)$  is nondegenerate, and $\phi^{-1}=\phi$.
\item Let $(\Omega,\star)$ be a group. Then $\eas(\Omega,\star)$ is nondegenerate.
indeed, for any $\alpha,\beta \in \Omega$,  $\phi(\alpha,\beta)=(\alpha\star\beta,\alpha)$, so $\phi$  is a bijection, of inverse
given by $\phi^{-1}(\alpha,\beta)=(\beta,\beta^{\star-1}\star \alpha)$. 
\item Let $(\Omega,\star)$ be an associative semigroup with the right inverse condition. Then $\eas'(\Omega,\star)$ 
is nondegenerate. Indeed, for any $\alpha,\beta \in \Omega$, 
$\phi(\alpha,\beta)=(\beta,\alpha\triangleright \beta)$, so $\phi$  is a bijection, 
of inverse given by $\phi^{-1}(\alpha,\beta)=(\beta\star \alpha,\alpha)$.
\end{enumerate}\end{example}

\subsection{EAS of cardinality two}

Here is a classification of EAS of cardinality two. The underlying set is $\Omega=\{a,b\}$
and the products will be given by two matrices
\begin{align*}
&\begin{pmatrix}
a\rightarrow a&a\rightarrow b\\
b\rightarrow a&b\rightarrow b
\end{pmatrix},&
\begin{pmatrix}
a\triangleright a&a\triangleright b\\
b\triangleright a&b\triangleright b
\end{pmatrix}.
\end{align*}
We shall use the two maps:
\begin{align*}
\pi_a&:\left\{\begin{array}{rcl}
\Omega&\longrightarrow&\Omega\\
\alpha&\longrightarrow&a,
\end{array}\right.&
\pi_b&:\left\{\begin{array}{rcl}
\Omega&\longrightarrow&\Omega\\
\alpha&\longrightarrow&b.
\end{array}\right.&
\end{align*}
We respect the indexation of EDS of \cite{Foissydendriforme}.

\begin{align*}
&\begin{array}{|c|c|c|l|}
\hline \mbox{Case}&\rightarrow&\triangleright&\mbox{Description}\\
\hline \hline \mathbf{A1}&\begin{pmatrix}
a&a\\a&a
\end{pmatrix}&
\begin{pmatrix}
a&a\\a&a
\end{pmatrix}&\eas(\Omega,\rightarrow,\pi_a)\\
\hline \mathbf{A2}&\begin{pmatrix}
a&a\\a&a
\end{pmatrix}&
\begin{pmatrix}
a&a\\b&b
\end{pmatrix}&\eas(\Omega,\rightarrow)\\
\hline \mathbf{C1} &\begin{pmatrix}
a&a\\a&b
\end{pmatrix}&
\begin{pmatrix}
a&a\\a&a
\end{pmatrix}&\eas(\Omega,\rightarrow,\pi_{\overline{0}})\\
\hline \mathbf{C3}&\begin{pmatrix}
a&a\\a&b
\end{pmatrix}&
\begin{pmatrix}
a&a\\b&b
\end{pmatrix}&\eas(\Z/2\Z,\times)\\
\hline \mathbf{C5} &\begin{pmatrix}
a&a\\a&b
\end{pmatrix}&
\begin{pmatrix}
b&b\\b&b
\end{pmatrix}&\eas(\Z/2\Z,\pi_{\overline{1}})\\
\hline \mathbf{C6}&\begin{pmatrix}
a&a\\a&b
\end{pmatrix}&
\begin{pmatrix}
a&a\\b&a
\end{pmatrix}&\\
\hline \mathbf{E1'-E2'}&\begin{pmatrix}
a&a\\b&b
\end{pmatrix}&
\begin{pmatrix}
a&a\\a&a
\end{pmatrix}&\eas(\Omega,\rightarrow,\pi_a)\\
\hline \mathbf{E3'}&\begin{pmatrix}
a&a\\b&b
\end{pmatrix}&
\begin{pmatrix}
a&a\\b&b
\end{pmatrix}&\eas(\Omega,\rightarrow)\\
\hline\mathbf{ F1}&\begin{pmatrix}
a&b\\a&b
\end{pmatrix}&
\begin{pmatrix}
a&a\\a&a
\end{pmatrix}&\eas(\Omega,\rightarrow,\pi_a)\\
\hline \mathbf{F3}&\begin{pmatrix}
a&b\\a&b
\end{pmatrix}&
\begin{pmatrix}
a&a\\b&b
\end{pmatrix}&\eas(\Omega)\\
\hline \mathbf{F4}&\begin{pmatrix}
a&b\\a&b
\end{pmatrix}&
\begin{pmatrix}
a&b\\b&a
\end{pmatrix}&\eas'(\Z/2\Z,+)\\
\hline \mathbf{H1}&\begin{pmatrix}
a&b\\b&a
\end{pmatrix}&
\begin{pmatrix}
a&a\\a&a
\end{pmatrix}&\eas(\Z/2\Z,+,\pi_{\overline{0}})\\
\hline \mathbf{H2}&\begin{pmatrix}
a&b\\b&a
\end{pmatrix}&
\begin{pmatrix}
a&a\\b&b
\end{pmatrix}&\eas(\Z/2\Z,+)\\
\hline \end{array}
\end{align*}
The nondegenerate EAS are $F_3$, $F_4$ and $H_2$.

\section{Generalized associative algebras}

\subsection{Discrete and linear versions}

\begin{defi}
Let $(\Omega,\rightarrow,\triangleright)$ be a set with two binary operations.
Let $(V,(*_\alpha)_{\alpha\in \Omega})$ be a family such that $V$ is a vector space and, for any $\alpha \in \Omega$,
$*_\alpha:V\otimes V\longrightarrow V$ is a linear map. We shall say that it is an $\Omega$-associative algebra if:
\begin{align}
\label{eq8}
&\forall x,y,z\in V,\:\forall \alpha,\beta \in \Omega,&
x*_\alpha (y *_\beta z)&=(x*_{\alpha \triangleright \beta} y)*_{\alpha \rightarrow \beta} z.
\end{align}
\end{defi}

\begin{example}
\begin{enumerate}
\item If $\Omega$ is a set, for $\eas(\Omega)$, (\ref{eq8}) becomes:
\begin{align*}
&\forall x,y,z\in V,\:\forall \alpha,\beta \in \Omega,&
x*_\alpha (y *_\beta z)&=(x*_\alpha y)*_\beta z.
\end{align*}
As a consequence, any linear span of $*_\alpha$ is associative. 
We recover the notion of matching associative algebra \cite{ZhangGaoGuo}.
\item If $(\Omega,\star)$ is a semigroup, for $\eas(\Omega,\star)$, (\ref{eq8}) becomes:
\begin{align*}
&\forall x,y,z\in V,\:\forall \alpha,\beta \in \Omega,&
x*_\alpha (y *_\beta z)&=(x*_\alpha y)*_{\alpha \star\beta} z.
\end{align*}
These are $(\Omega,\star)$-family associative algebras.
\end{enumerate}
\end{example}

\begin{remark}
This does not include the multiassociative and the dual multiassociative algebras introduced by Giraudo in \cite{Giraudo}.
We shall see that the dimension of the $n$-th components of the operad of $\Omega$-associative algebras
is $|\Omega|^{n-1}$ for any $n\geqslant 2$, whereas it is constant for dual multiassociative algebras and described
by Narayana numbers for dual multiassociative algebras.
\end{remark}

In order to linearize these axioms, let us first consider the following lemma, proved in \cite{FoissyPrelie2}:

\begin{lemma}\label{lem2.6}
Let $(\Omega,\rightarrow,\triangleright)$ be a set with two binary operations. We consider the maps
\begin{align*}
\phi&:\left\{\begin{array}{rcl}
\Omega^2&\longrightarrow&\Omega^2\\
(\alpha,\beta)&\longrightarrow&(\alpha \rightarrow \beta,\alpha \triangleright \beta),
\end{array}\right.&
\tau&:\left\{\begin{array}{rcl}
\Omega^2&\longrightarrow&\Omega^2\\
(\alpha,\beta)&\longrightarrow&(\beta,\alpha).
\end{array}\right.
\end{align*}
Then $(\Omega,\rightarrow,\triangleright)$ is an EAS if, and only if:
\begin{align}
\label{eq23}
(\id \times \phi)\circ (\phi\times \id)\circ (\id \times \phi)&=(\phi \times \id)\circ (\id \times \tau)\circ (\phi\times \id).
\end{align}
\end{lemma}

This naturally leads to the following definition:

\begin{defi}
Let $A$ be a vector space and let $\Phi:A\otimes A\longrightarrow A\otimes A$ be a linear map.
We shall say that $(A,\Phi)$ is a linear extended associative semigroup (briefly, $\ell$EAS) if:
\begin{align}
\label{eq29}
(\id \otimes \Phi)\circ (\Phi\otimes \id)\circ (\id \otimes \Phi)&=(\Phi \otimes \id)\circ (\id \otimes \tau)\circ (\Phi\otimes \id),
\end{align}
where $\tau:A\otimes A\longrightarrow A\otimes A$ is the usual flip:
\[\tau:\left\{\begin{array}{rcl}
A\otimes A&\longrightarrow& A\otimes A\\
a\otimes b&\longrightarrow&b\otimes a.
\end{array}\right.\]
We shall say that $(A,\Phi)$ is nondegenerate if $\Phi$ is invertible.
\end{defi}

\begin{example}
\begin{enumerate}
\item Let $(\Omega,\rightarrow,\triangleright)$ be an EAS and let $A=\K\Omega$ be its algebra, that is to say the vector space
generated by $\Omega$. We define:
\[\Phi:\left\{\begin{array}{rcl}
\K\Omega\otimes \K\Omega&\longrightarrow&\K\Omega\otimes \K\Omega\\
\alpha\otimes \beta&\longrightarrow&(\alpha \rightarrow \beta)\otimes (\alpha \triangleright \beta),
\end{array}\right.\]
where $\alpha,\beta \in \Omega$. Lemma \ref{lem2.6} implies that $(\K\Omega,\Phi)$ is an $\ell$EAS,
which we call the linearization of $(\Omega,\rightarrow,\triangleright)$. 
\item Here are examples of $\ell$EAS of dimension 2, which are not linearization of an EAS.
In these examples, $A$ is a two-dimensional space with basis $(x,y)$,
and  the maps $\Phi$ are  given by their matrices in the basis $(x\otimes x,x\otimes y,y\otimes x,y\otimes y)$.
\begin{align*}
&\begin{pmatrix}0&0&1&0\\0&0&0&0\\0&0&0&0\\0&0&0&0\end{pmatrix},&
&\begin{pmatrix}0&0&0&0\\0&0&\lambda&0\\0&0&0&0\\0&0&0&0\end{pmatrix},&
&\begin{pmatrix}1&0&0&0\\0&0&0&0\\0&0&0&0\\0&0&0&0\end{pmatrix},&
&\begin{pmatrix}1&0&0&0\\0&0&1&0\\0&0&0&0\\0&0&0&0\end{pmatrix},\\
&\begin{pmatrix}1&0&0&0\\0&0&0&0\\0&1&0&0\\0&0&0&0\end{pmatrix},&
&\begin{pmatrix}1&0&0&0\\0&0&0&0\\0&0&1&0\\0&0&0&0\end{pmatrix},&
&\begin{pmatrix}1&0&0&0\\0&0&0&0\\0&0&0&0\\0&0&1&0\end{pmatrix},&
&\begin{pmatrix}1&0&0&0\\0&0&1&0\\0&0&1&0\\0&0&0&0\end{pmatrix},\\
&\begin{pmatrix}1&0&0&0\\0&0&0&0\\0&1&0&0\\0&0&1&0\end{pmatrix},&
&\begin{pmatrix}1&0&0&0\\0&0&0&0\\0&1&1&0\\0&0&0&0\end{pmatrix},&
&\begin{pmatrix}1&0&0&0\\0&0&1&0\\0&0&1&0\\0&0&1&0\end{pmatrix},&
&\begin{pmatrix}1&0&0&0\\0&0&0&0\\0&1&0&0\\0&1&-1&0\end{pmatrix},\\
&\begin{pmatrix}1&0&0&0\\0&0&0&0\\0&0&0&0\\0&0&0&1\end{pmatrix},&
&\begin{pmatrix}1&0&0&0\\0&0&0&0\\0&1&0&0\\0&0&0&1\end{pmatrix},&
&\begin{pmatrix}1&1&0&0\\0&0&0&0\\0&0&0&0\\0&0&1&1\end{pmatrix},&
&\begin{pmatrix}1&0&0&0\\0&0&0&0\\0&1&0&0\\0&0&1&1\end{pmatrix},\\
&\begin{pmatrix}1&0&1&0\\0&0&-1&0\\0&1&-1&0\\0&0&2&1\end{pmatrix},
\end{align*}
where $\lambda$ is a scalar. More details on these examples car be found in \cite{FoissyEAS}.
\end{enumerate} \end{example}

\begin{notation}
Let $(A,\Phi)$ be a pair, such that $A$ is a vector space and $\Phi:A\otimes A\longrightarrow A\otimes A$
is a linear map. We use the Sweedler notation:
\[\Phi(a\otimes b)=\sum a'\rightarrow b' \otimes a''\triangleright b''.\]
Note that the operations $\rightarrow$ and $\triangleright$ may not exist, nor the coproducts $a'\otimes a''$
or $b'\otimes b''$.
With this notation, (\ref{eq29}) can be rewritten as:
\begin{align}
\tag{\ref{eq29}'}
&\sum\sum\sum a' \rightarrow (b' \rightarrow c')'\otimes 
 (a'' \triangleright (b' \rightarrow c')'')'\rightarrow (b'' \triangleright c'')'
 \otimes  (a'' \triangleright (b' \rightarrow c'))''\triangleright (b'' \triangleright c'')''\\
\nonumber &=\sum\sum 
 (a' \rightarrow b')'\rightarrow c' \otimes (a' \rightarrow b')''\triangleright c''
 \otimes a'' \triangleright b''.
\end{align}\end{notation}

\begin{defi}\label{defialgebras}
Let $V$ be a vector space and $*$ a linear map:
\[*:\left\{\begin{array}{rcl}
A&\longrightarrow&\hom(V\otimes V,V)\\
a&\longrightarrow&*_a:\left\{\begin{array}{rcl}
V\otimes V&\longrightarrow&v\\
x\otimes y&\longrightarrow&x*_a y.
\end{array}\right.
\end{array}\right.\]
We shall say that $(V,*)$ is a $\Phi$-associative algebra if:
\begin{align}
&\forall x,y,z\in V,\:\forall a,b\in A,&
x*_a (y *_b z)&=\sum (x*_{a''\triangleright b''} y)*_{a' \rightarrow b'} z.
\end{align}
We shall say that $(V,*)$ is an opposite $\Phi$-associative algebra if:
\begin{align}
&\forall x,y,z\in V,\:\forall a,b\in A,&
\sum x*_{a'\rightarrow b'} (y *_{a''\triangleright b''} z)&=(x*_b y)*_a z.
\end{align}
\end{defi}

\begin{remark}\begin{enumerate}
\item We define $*^{op}$ by $x *^{op}_a y=y*_a x$. Then $(V,*)$ is a $\Phi$-associative algebra
if, and only if, $(V,*^{op})$ is an opposite $\Phi$-associative algebra.
\item If $\Phi$ is invertible, opposite $\Phi$-associative algebras and $\Phi^{-1}$-associative algebras are the same.
\end{enumerate}\end{remark}

\begin{remark}
If $(\Omega,\rightarrow,\triangleright)$ is an EAS and if $(\K\Omega,\Phi)$ is its linearization,
then the categories of $\Omega$-associative algebras and of $\Phi$-associative algebras are isomorphic:
if $(V,(*_\alpha)_{\alpha\in \Omega})$ is an $\Omega$-algebra, then we obtain a $\Phi$-associative algebra with the map:
\[*:\left\{\begin{array}{rcl}
\K\Omega&\longrightarrow&\hom(V\otimes V,V)\\
\alpha \in \Omega&\longrightarrow&*_\alpha.
\end{array}\right.\]
In this way, $\Omega$-associative algebras can be seen as particular examples of $\Phi$-associative algebras.
\end{remark}

\subsection{Free objects}

\begin{notation}
Let $V$ be a vector space.  We put:
\[T_A(V)=\bigoplus_{n=1}^\infty A^{\otimes (n-1)}\otimes V^{\otimes n}.\]
If $a_1,\ldots, a_{n-1}\in A$, $x_1,\ldots,x_n \in V$, we shall denote their tensor product in $A^{\otimes (n-1)}\otimes V^{\otimes n}$
by $a_1\ldots a_{n-1}x_1\ldots x_n$. Such a tensor will be called an \emph{$A$-typed word of length} $n$. 
We shall use the following map:
\[\cdot:\left\{\begin{array}{rcl}
T_A(V)\otimes A\otimes V&\longrightarrow&T_A(V)\\
a_1\ldots a_{n-1}x_1\ldots x_n\otimes a\otimes x&\longrightarrow&
a_1\ldots a_{n-1}x_1\ldots x_n\cdot ax=aa_1\ldots a_{n-1}ax_1\ldots x_nx.
\end{array}\right.\]
\end{notation}

\begin{theo}\label{theo3.2}
For any vector space $V$, we define bilinear products $*_a$ on $T_A(V)$ in the following way,
by induction on the length of $A$-typed words:
\begin{align*}
w*_a z&=w\cdot a z,&
u*_a (v\cdot b z)&=\sum (u*_{a'' \triangleright b''} v)\cdot (a' \rightarrow b')z,
\end{align*}
where $u,v,w\in T_A(V)$, $z\in V$ and $a,b \in A$. 
The following conditions are equivalent:
\begin{enumerate}
\item $(A,\Phi)$ is an $\ell$EAS.
\item For any vector space $V$, $(T_A(V),*)$ is a $\Phi$-associative algebra.
\item There exists a nonzero vector space $V$ such that $(T_A(V),*)$ is a $\Phi$-associative algebra.
\end{enumerate}
Moreover, if these conditions hold, then $(T_A(V),*)$ is the free $\Phi$--associative algebra generated by $V$.
\end{theo}

\begin{proof} Obviously, $2.\Longrightarrow 3$. \\

$3.\Longrightarrow 1$. Let $x,y,z,t$ be four nonzero elements of any nonzero vector space $V$. 
For any $a,b,c\in A$:
\begin{align*}
x *_a (y*_b c zt)&=
\sum\sum\sum (a'' \triangleright (b' \rightarrow c'))''\triangleright (b'' \triangleright c'')''\\
&\hspace{2cm}(a'' \triangleright (b' \rightarrow c')'')'\rightarrow (b'' \triangleright c'')'\\
&\hspace{2cm}a' \rightarrow (b' \rightarrow c')' xyzt,\\
\sum (x*_{a' \triangleright b'} y)*_{a'' \rightarrow b''} c zt
&=\sum\sum (a'' \triangleright b'')((a' \rightarrow b')''\triangleright c'')((a' \rightarrow b')'
\rightarrow c')  xyzt.
\end{align*}
This immediately gives (\ref{eq29}).\\

$1.\Longrightarrow 2$. Let us prove that for any $A$-typed words $u,v,w$, for any $a,b \in A$,
\[u*_a (v *_b w)=\sum(u*_{a' \triangleright b'} v)*_{a'' \rightarrow b''} w.\]
We proceed by induction on the length $n$ of $w$. If $n=1$, we put $w=z\in V$. 
\begin{align*}
u*_a (v *_b w)&=u*_a (v\cdot b z)=\sum(u*_{a' \triangleright b'} v)\cdot (a'' \rightarrow b'') z
=\sum(u*_{a' \triangleright b'} v)*_{a'' \rightarrow b''} z.
\end{align*}
Let us assume the result at rank $n-1$. We put $w=w'\cdot c z$. 
\begin{align*}
u*_a (v *_b w)&=\sum u*_a((v*_{b' \triangleright c'}w')\cdot (b'' \rightarrow c'')z)\\
&=\sum\sum u *_{a'' \triangleright(b' \rightarrow c')''}(v*_{b'' \triangleright c''} w')\cdot 
a'\rightarrow(b' \rightarrow c')'' z\\
&=\sum\sum\sum(u*_{(a'' \triangleright(b' \rightarrow c')'')'' \triangleright (b'' \triangleright c'')''}v)
*_{(a'' \triangleright(b' \rightarrow c')'')' \rightarrow (b'' \triangleright c'')'} w'
\cdot a'\rightarrow(b' \rightarrow c')' z\\
&=\sum\sum(u*_{a'' \triangleright b''}v)*_{(a' \rightarrow b')'' \triangleright c''} w'
\cdot (a'\rightarrow b')' \rightarrow c' z\\
&=\sum(u*_{a'' \triangleright b''}v)*_{a' \rightarrow b'} (w'\cdot c z)\\
&=\sum(u*_{a'' \triangleright b''}v)*_{a' \rightarrow b'} w.
\end{align*}
We use the induction hypothesis for the third equality and (\ref{eq29})  for the fourth one.\\

\textit{Freeness}. Let us assume that conditions 1 and 2 hold. Let $W$ be an $A$-associative algebra,
and $\theta:V\longrightarrow W$ be any linear map. Let us prove that it can be uniquely extended as
a $\Phi$-associative algebra morphism $\Theta$ from $T_A(V)$ to $W$. \\

\textit{Existence of $\Theta$}. We inductively define $\Theta(w)$ for any $A$-typed word $w$ by induction on its length $n$.
If $n=1$, then $w\in V$ and we put $\Theta(w)=\theta(w)$. Otherwise, let us write $w=w'\cdot a z$. We put:
\[\Theta(w)=\Theta(w')*_a\theta(z).\]
Let $u,v$ be two $A$-typed words, and let us prove that for any $a \in A$,
$\Theta(u*_a v)=\Theta(u)*_a \Theta(v)$. We proceed by induction on the length $n$ of $v$.
If $n=1$, then by definition of $\Theta$, this is true. Otherwise, let us put $v=v'\cdot b z$. Then:
\begin{align*}
\Theta(u*_a v)&=\sum \Theta((u*_{a'' \triangleright b''} v')\cdot (a' \rightarrow b') z)\\
&=\sum\Theta(u*_{a'' \triangleright b''} v')*_{a' \rightarrow b'} \theta(z)\\
&=\sum(\Theta(u)*_{a'' \triangleright b''}\Theta( v'))*_{a' \rightarrow b'} \theta(z)\\
&=\Theta(u)*_a (\Theta(v')*_b \theta(z))\\
&=\Theta(u)*_a \Theta(v).
\end{align*}
So $\Theta$ is a $\Phi$-associative algebra morphism.\\

\textit{Uniqueness of $\Theta$}. If $\Theta'$ is another morphism extending $\theta$, for any $A$-typed word $u$, for any $a\in A$,
for any $z\in V$:
\[\Theta'(u\cdot a z)=\Theta'(u*_a z)=\Theta'(u)*_a \theta(z).\]
An easy induction on the length proves that for any $A$- typed word $u$, $\Theta'(u)=\Theta(u)$. \end{proof}

\subsection{Links with associative algebras}

\begin{prop} \label{prop3.3}
Let $(A,\Phi)$ such that $A$ is a vector space and $\Phi:A\otimes A\longrightarrow A\otimes A$
is a linear map. We assume that $V$ is a vector space and $*:A\longrightarrow\hom(V\otimes V,V)$ is a linear map. 
We define a product on $A\otimes V$ by:
\begin{align*}
&\forall x,y\in V,\:\forall a,b \in A,&
xa \star yb&=\sum x *_{a'' \triangleright b''} y a' \rightarrow b'.
\end{align*}
Then:
\begin{enumerate}
\item If $(A,\Phi)$ is an $\ell$EAS and $(V,*)$ is
a $\Phi$-associative algebra, then $(A\otimes V,\star)$ is an associative algebra.
\item  If $(A,\Phi)$ is a nondegenerate $\ell$EAS and $(A\otimes V,\star)$ is an associative algebra,
then $(V,*)$ is a $\Phi$-associative algebra.
\item Let $V$ be a nonzero vector space. If $(A\otimes T_A(V),\star)$ is an associative algebra, then $(A,\Phi)$ is an $\ell$EAS.
\end{enumerate}
\end{prop}

\begin{proof}
Let $a,b,c \in A$ and $x,y,z\in V$. In $A\otimes V$:
\begin{align}
\label{EQ9}
xa \star(yb\star zc)&=\sum\sum x*_{a'' \triangleright (b' \rightarrow c')''} (y*_{b'' \triangleright c''}z)
a' \rightarrow (b' \rightarrow c')',\\
\nonumber (xa\star yb)\star zc&= \sum\sum (x*_{a'' \triangleright b''}y)*_{(a' \rightarrow b')''\triangleright c''} z
(a' \rightarrow b')'\rightarrow c'.
\end{align}

1. We put 
\[x=\sum\sum a' \rightarrow (b' \rightarrow c')'\otimes a'' \triangleright (b' \rightarrow c')''
\otimes b'' \triangleright c''=(\Phi\otimes \id)\circ (\id \otimes \Phi)(a\otimes b\otimes c).\]
By (\ref{eq29}):
\begin{align*}
(\id \otimes \Phi)(x)&=(\Phi\otimes \id_A)\circ (\id_A \otimes \tau)\circ (\Phi\otimes \id_A)(a\otimes b\otimes c)\\
&=\sum\sum (a' \rightarrow b')'\rightarrow c'\otimes (a' \rightarrow b')''\triangleright c''
\otimes a'' \triangleright b''.
\end{align*}
As $V$ is $\Phi$-associative, we obtain that $\star$ is associative.\\

2. By composition, the following map is bijective:
\[\Psi=(\Phi\otimes \id_A)\circ (\id_A\times \Phi):
\left\{\begin{array}{rcl}
V^{\otimes 3}&\longrightarrow&V^{\otimes 3}\\
a\otimes b\otimes c&\longrightarrow&a' \rightarrow(b' \rightarrow c')'\otimes
a'' \triangleright (b' \rightarrow c')'' \otimes b'' \triangleright c''.
\end{array}\right.\]
Let $a\otimes b\otimes c\in A^{\otimes 3}$ and $a_1\otimes b_1\otimes c_1=\Psi^{-1}(a\otimes b\otimes c)$. 
For any $x,y,z\in A$:
\begin{align*}
xa_1 \star (yb_1 \star zc_1)&=\left(x*_b (y*_c z)\right)a,\\
(xa_1\star yb_1)\star zc_1&=\sum\left((x*_{b'' \triangleright c''} y)*_{b' \rightarrow c'} z \right)a.
\end{align*}
The associativity of $\star$ induces the axiom of $\Phi$-associative algebra for $V$.\\

3. Let $x,y,z\in V$, nonzero (not necessarily distinct). 
From the associativity of $\star$, we immediately deduce from (\ref{EQ9}) that:
\begin{align*}
&\sum \sum \sum
a' \rightarrow (b' \rightarrow c')'\otimes 
(a'' \triangleright (b' \rightarrow c')'')'\rightarrow (b'' \triangleright c'')'\otimes
(a'' \triangleright (b' \rightarrow c')'')''\triangleright (b'' \triangleright c)''\\
&=\sum\sum(a' \rightarrow b')' \rightarrow c'\otimes
(a' \rightarrow b')''\triangleright c''\otimes a'' \triangleright b''.
\qedhere \end{align*}
So $(A,\Phi)$ is an $\ell$EAS.
\end{proof}

\begin{remark}
As a corollary, if $(\Omega,\rightarrow,\triangleright)$ is an EAS, then $\Omega$-associative algebras
are 2-parameters associative algebras with $*_{\alpha,\beta}=*_{\alpha \triangleright \beta}$. 
This will be formalized in Proposition \ref{prop3.14} by an operad morphism.
\end{remark}

\begin{prop}\label{prop3.4}
Let $(A,\Phi)$ be an $\ell$EAS and let $V$ be a nonzero vector space.
\begin{enumerate}
\item The following conditions are equivalent:
\begin{enumerate}
\item The associative algebra $T_A(V)\otimes A$ is generated by $V\otimes A$.
\item $\Phi$ is surjective.
\end{enumerate}
\item The following conditions are equivalent:
\begin{enumerate}
\item The subalgebra $T_A(V)\otimes A$ generated by $V\otimes A$ is free.
\item $\Phi$ is injective.
\end{enumerate}
\end{enumerate}
\end{prop}

\begin{proof}
We denote by $W$ the subalgebra of $T_A(V)\otimes A$ generated by $V\otimes A$.
Note that it is graded by the length of words.\\

1. $(a)\Longrightarrow (b)$. Let $a\otimes b\in A^{\otimes 2}$. Let us choose a nonzero element $x$ of $V$.
Then $xxa b \in A$. Because of the graduation, we can write this element under the form:
\[xxa b=\sum_{i=1}^n x_i a_i\star y_i b_i= \sum_{i=1}^n \sum x_i y_i (a_i''\triangleright b''_i)
(a_i'\rightarrow b_i').\]
Applying an element $f$ of $V^*$ such that $f(x)=1$, we obtain
\[\Phi\left(\sum_{i=1}^n f(x_i)f(y_i)a_i\otimes b_i\right)=a\otimes b,\]
 so $\Phi$ is surjective.\\

1. $(b)\Longrightarrow (a)$. Let $x_1\ldots x_n a_1\ldots a_n$ be a word of length $n$,
and let us prove that it belongs to $W$ by induction on $n$. This is obvious if $n=1$. 
Otherwise, there exists $x=\sum b_{n-1}\otimes b_n \in A^{\otimes 2}$, such that 
\[\Phi\left(\sum b_{n-1}\otimes b_n\right)=a_n \otimes a_{n-1}.\] 
By the induction hypothesis, $x_1\ldots x_{n-1}a_1\ldots a_{n-2}b_{n-1}\in W$,
so:
\begin{align*}
\sum x_1\ldots x_{n-1}a_1\ldots a_{n-2}b_{n-1}\star x_n b_n
&=\sum \sum x_1\ldots x_n a_1\ldots a_{n-2}(b''_{n-1}\triangleright b''_n)(b'_{n-1}\rightarrow b'_n)\\
&=x_1\ldots x_n a_1\ldots a_n \in W.
\end{align*}

2. $(a)\Longrightarrow (b)$. Because of the graduation, $W$ is freely generated by $V\otimes A$.
Let $x$ be a nonzero element of $V$.
If $\sum a_n\otimes b_n\neq 0$, by freeness, $\sum xa_n\star xb_n\neq 0$ and:
\begin{align*}
\sum xa_n\star xb_n&=\sum\sum xx(a_n'' \triangleright b_n'')(a_n' \rightarrow b_n')\neq 0,
\end{align*}
So $\Phi\left(\sum a_n \otimes b_n\right)\neq 0$.\\

2. $(b)\Longrightarrow (a)$. We shall use the following map:
\[\Phi'=\tau \circ \Phi:\left\{\begin{array}{rcl}
A\otimes A&\longrightarrow&A\otimes A\\
a\otimes b&\longrightarrow&\sum a'' \triangleright b'' \otimes a' \rightarrow b'.
\end{array}\right.\]
As $\Phi$ is injective, $\Phi'$ is injective. Let $x_1,\ldots,x_n\in V$ and let $a_1,\ldots,a_n\in A$.
An easy induction on $n$ proves that:
\begin{align*}
&x_1a_1\star \ldots\star x_na_n\\
&=x_1\ldots x_n \left(\id_A^{\otimes (n-2)}\otimes \Phi'\right)\circ
\left(\id_A^{\otimes (n-3)}\otimes \Phi'\otimes \id_A\right)\circ \ldots
\circ \left(\Phi'\otimes\id_A^{\otimes (n-2)}\right)(a_1\otimes \ldots \otimes a_n).
\end{align*}
As a consequence, the following algebra map is injective:
\[\left\{\begin{array}{rcl}
T(V\otimes A)&\longrightarrow&T_A(V)\otimes A\\
x_1a_1\ldots x_na_n&\longrightarrow&x_1a_1\star\ldots \star x_n a_n.
\end{array}\right.\]
So the  image of this morphism, which is $W$,  is freely generated by $V\otimes A$. \end{proof}

\begin{remark}
Consequently, for any vector space $V$, $(T_A(V),\star)$ is freely generated by $V\otimes A$
if, and only if, $(A,\Phi)$ is nondegenerate.
\end{remark}

\subsection{Operadic aspects and Koszul duality}

In this section, $(A,\Phi)$ is an $\ell$EAS.\\

\begin{notation}
We denote  the nonsymmetric operad of $\Phi$-associative algebras by $\as_\Phi$,
and the nonsymmetric operad of opposite $\Phi$-associative algebras by $\as_\Phi'$.
In other words, $\as_\phi$ is the nonsymetric operad generated by $A=\as_\phi(2)$, with the relations
\[a\circ_2 b=\sum a'\rightarrow b'\circ_1 a''\triangleright b'',\]
whereas $\as_\phi'$ is the nonsymetric operad generated by $A=\as_\phi'(2)$, with the relations
\[a\circ_1 b=\sum a'\rightarrow b'\circ_2 a''\triangleright b''.\]
We denote by $\sym\as_\Phi$, respectively by $\sym \as_\Phi'$, the operad of $\Phi$-associative algebras,
respectively of opposite $\Phi$-associative algebras.
\end{notation}

\begin{remark}\label{remark5}
\begin{enumerate}
\item $\sym \as_\phi$, respectively $\sym \as_\Phi'$, is the symmetrisation of the nonsymmetric operad
$\as_\Phi$, respectively $\as_\phi'$.
\item If $\Phi$  is nondegenerate, then $\as_\Phi'=\as_{\Phi^{-1}}$.
\item $\sym \as_\Phi$ and $\sym\as_\Phi'$ are isomorphic operads, through the morphism
\[\left\{\begin{array}{rcl}
\sym\as_\Phi&\longrightarrow&\sym \as_\Phi'\\
a\in A&\longrightarrow&a^{op}=a^{(12)}.
\end{array}\right.\]
\end{enumerate}\end{remark}

From the description of free $\Phi$-associative algebras, we obtain a combinatorial description of $\as_\Phi$:

\begin{prop}\label{prop3.5}
For any $n\geqslant 1$, $\as_\Phi(n)$ is the vector space $A^{\otimes (n-1)}$. 
For any $a_k\ldots a_1 \in A^{\otimes k}=\as_\Phi(k+1)$,
for any $b_l\ldots b_1\in A^{\otimes l}=\as_\Phi(l+1)$, for any $i\in [k+1]$:
\begin{align*}
&a_k\ldots a_1\circ_i b_l\ldots b_1\\
&=\begin{cases}
b_l\ldots b_1 a_k\ldots a_1\mbox{ if }i=1,\\
a_k\ldots a_i (\Phi\otimes \id^{\otimes (l-2)})\circ \ldots
\circ (\id\otimes \Phi\otimes \id^{\otimes (l-3)})
\circ (\id^{\otimes (l-2)}\otimes \Phi)(a_{i-1}b_l\ldots b_1)a_{i-2}\ldots a_1\\
\hspace{1cm}\mbox{ if }i\geqslant 2.
\end{cases}
\end{align*}
\end{prop}

\begin{example}
Let us consider linearizations of EAS.
\begin{enumerate}
\item  For $\eas(A,\star)$, this simplifies as:
\begin{align*}
\alpha_1\ldots \alpha_k\circ_i \beta_1\ldots \beta_l&=\alpha_1\ldots \alpha_{i-1}
(\alpha_{i-1}\star \beta_1)\ldots (\alpha_{i-1}\star \beta_l)\alpha_i\ldots \alpha_k.
\end{align*}
\item For $\eas(\Omega)$, this simplifies as:
\begin{align*}
\alpha_1\ldots \alpha_k\circ_i \beta_1\ldots \beta_l&=\alpha_1\ldots \alpha_{i-1}
\beta_1\ldots \beta_l\alpha_i\ldots \alpha_k.
\end{align*}
This operad is used in \cite{CombeGiraudo}. When $\Omega$ has two elements, this gives the operad 
of duplexes of vertices of cubes defined in \cite[Section 6.3]{Pirashvili}.
\item If $(A,\star)$ is a group, we obtain for $\eas'(A,\star)$:
\begin{align*}
\alpha_1\ldots \alpha_k\circ_i \beta_1\ldots \beta_l&=\alpha_1\ldots \alpha_{i-2}
\left(\alpha_{i-1}\star \beta_l^{-1}\star \ldots\star\beta_1^{-1}\right)\beta_1\ldots \beta_l \alpha_i\ldots \alpha_k.
\end{align*}\end{enumerate}\end{example}

\begin{prop}\label{prop3.6}
Let us assume that $A$ is finite-dimensional.
\begin{enumerate}
\item Koszul dual of the nonsymmetric operad $\as_\Phi$ is isomorphic to $\as_{\Phi^*}'$.
\item Koszul dual of the nonsymmetric operad $\as_\Phi'$ is isomorphic to $\as_{\Phi^*}$.
\item Koszul dual of the  operad $\sym \as_\Phi$ is isomorphic to $\sym \as_{\Phi^*}$.
\end{enumerate}
\end{prop}

\begin{proof}
\begin{enumerate}
\item We identify $\as_\Phi(2)^*=A$ and $A^*$. This identification induces a pairing between the free nonsymmetric
operad $\bfF_A$ generated by $A$ and the free nonsymmetric operad $\bfF_{A^*}$ generated by $A^*$.
In particular, if $a,b\in A$, $f,g\in A^*$,
\begin{align*}
\langle f\circ_1 g,a\circ_1 b\rangle&=f(a)g(b),&
\langle f\circ_2 g,a\circ_2 b\rangle&=-f(a)g(b),\\
\langle f\circ_1 g,a\circ_2 b\rangle&=0,&
\langle f\circ_2 g,a\circ_1 b\rangle&=0.
\end{align*}
We denote by $I$ the space of relations of $\as_\Phi(3)$: this is the subspace of $\bfF_A$ generated by the elements
\[\sum a'\rightarrow b'\circ_1 a''\triangleright b''-a\circ_2 b,\]
with $a,b\in A$. Note that $\as_\Phi^!$ is the quotient of $\bfF_{A^*}$ by the operadic ideal generated by $I^\perp$.
We also denote by $I'$ the space of relations of $\as_{\Phi^*}(3)$: this is the subspace of $\bfF_{A^*}$ generated by the elements
\[\sum f'\rightarrow g'\circ_2 f''\triangleright g''-f\circ_1 g,\]
with $f,g\in A^*$. Let $a,b\in A$ and $f,g\in A^*$. 
\begin{align*}
&\langle \sum f'\rightarrow g'\circ_2 f''\triangleright g''-f\circ_1 g,\sum a'\rightarrow b'\circ_1 a''\triangleright b''-a\circ_2 b\rangle\\
&=-\Phi^*(f\otimes g)(a\otimes b)-(f\otimes g)(\Phi(a\otimes b))\\
&=0,
\end{align*}
so $I'\subseteq I^\perp$. Moreover:
\begin{align*}
\dim(\bfF_A(3))&=2\dim(A)^2, &\dim(I)&=\dim(I')=\dim(A)^2,
\end{align*}
so $\dim(I^\perp)=2\dim(A)^2-\dim(A)^2=\dim(A)^2=\dim(I')$ and finally $I^\perp=I'$.
\item By duality.
\item By symmetrisation, $(\sym\as_\Phi)^!=\sym\as_{\Phi^*}'$, which  is isomorphic to $\sym\as_{\Phi^*}$, 
see Remark \ref{remark5}. \qedhere
\end{enumerate}
\end{proof}

\begin{example}
Let $\Omega$ be a finite EAS. Koszul dual of the operad $\as_A$ of $\Omega$-associative algebra
is generated by the products $\star_\alpha$, with $\alpha\in \Omega$, and the relations
\begin{align*}
&\forall \alpha,\beta \in \Omega,&
\left(\sum_{\substack{(\alpha',\beta')\in \Omega^2,\\ \phi(\alpha',\beta')=(\alpha,\beta)}}
\star_{\alpha'}\circ (I,\star_{\beta'})\right)=\star_\alpha(\star_\beta,I).
\end{align*}
\end{example}

\begin{theo}\label{theo3.7}
If $A$ is finite-dimensional, the nonsymmetric operads $\as_{\Phi}$ and $\as'_{\Phi}$
as well as the operad $\sym\as_\Phi$ are Koszul.
\end{theo}

\begin{proof} We shall use the rewriting method of \cite{Dotsenko,LodayVallette}. 
We shall write elements of the free nonsymmetric operad generated by $\as_A(2)$ as planar trees which vertices
are decorated by elements of $A$. The rewriting rules are:
\[\bdtroisdeux(a,b)\longrightarrow \sum\bdtroisun(a'' \rightarrow b'',a'' \triangleright b'')\]
for any $a,b \in A$. There is only one family of critical monomials, which are the trees
\[\bdquatrequatre(a,b,c)\]
with $a,b,c \in A$.  Koszularity of $\as_A$ comes from the confluence of the following diagram:
\begin{align}
\label{EQ10}
\xymatrix{&T_1\ar[rd] \ar[ld]&\\
T_2\ar[ddr]&&T_3 \ar[d]\\
&&T_4\ar[ld]\\
&T_5}
\end{align}
with:
\begin{align*}
T_1&=\bdquatrequatre(a,b,c),\\
T_2&=\sum\bdquatrecinq(a' \rightarrow b',a'' \triangleright b'',c),\\
T_3&=\sum\bdquatretrois(a,b' \rightarrow c',b'' \triangleright c''),\\
T_4&=\sum\sum\bdquatredeux(a' \rightarrow (b' \rightarrow c')',a'' \triangleright (b' \rightarrow c')'',
b''\triangleright c'')\\
T_5&=\sum\sum\bdquatreun((a' \rightarrow b')'\rightarrow c',(a' \rightarrow b')''\triangleright c'', 
a'' \triangleright b'')\\
&=\sum\sum\sum\bdquatreun(a' \rightarrow (b'\rightarrow c')', (a'' \triangleright (b' \rightarrow c')'')'
\rightarrow (b''\triangleright c'')', (a'' \triangleright (b' \rightarrow c')'')''
\triangleright (b''\triangleright c'')'').
\end{align*}
The equality between the two expressions of  $T_5$ is equivalent to (\ref{eq29}).
\end{proof}

Here is another application of Diagram (\ref{EQ10}):

\begin{prop}
Let $\bfP$ be a nonsymmetric set operad such that for any $n\geqslant 1$, 
the following map is a linear isomorphism:
\[\iota_n:\left\{\begin{array}{rcl}
\bfP(2)^{\otimes (n-1)}&\longrightarrow&\bfP(n)\\
p_1\otimes \ldots \otimes p_{n-1}&\longrightarrow&p_1\circ_1(p_2\circ_1(\ldots \circ_1(p_{n-2}\circ_1p_{n-1})\ldots)).
\end{array}\right.\]
Then there exists an $\ell$EAS $(A,\Phi)$ such that $\bfP$ is isomorphic to $\as_\Phi$. 
\end{prop}

\begin{proof}
We put $A=\bfP(2)$ as a vector space. As $\iota_3$ is bijective, for any $a\otimes b\in A\otimes A$,
there exists a unique $\Phi(a\otimes b)=\sum a'\rightarrow b'\otimes a''\triangleright b'' \in A\otimes A$ such that
\[a \circ_2 b=\sum (a' \rightarrow b')\circ_1 (a'' \triangleright b''),\]
or, equivalently:
\[\bdtroisdeux(a,b)=\sum \bdtroisun(a' \rightarrow b',a'' \triangleright b'').\]
For any $a,b,c \in A$, let us compute $a \circ_2 (b\circ_2 c)$ into two different ways.
This element is the tree $T_1$ of (\ref{EQ10}), and, following the two paths of this diagram, we obtain
that in $\bfP(3)$:
\begin{align*}
&\sum \sum (a' \rightarrow b')'\rightarrow c'\circ_1 ((a' \rightarrow b')''\triangleright c''\circ_1
(a'' \triangleright b''))\\
&=\sum\sum  \sum a' \rightarrow (b'\rightarrow c')'\circ_1  (a'' \triangleright (b' \rightarrow c')'')'
\rightarrow (b''\triangleright c'')'\circ( (a'' \triangleright (b' \rightarrow c')'')''
\triangleright (b''\triangleright c'')'')). 
\end{align*}
As $\iota_4$ is an isomorphism, we obtain the axioms of $\ell$EAS for $(A,\Phi)$. 
Hence, we obtain an operad isomorphism from $\as_\Phi$ to $\bfP$, sending $*_a$ to $a$ for any $a \in A$.  \end{proof}

\subsection{Associative products}

We now look for operad morphisms from the operad of associative algebras to the operad $\sym\as_\Phi$,
where $(A,\Phi)$ is an $\ell$EAS, or equivalently to products $m\in \sym\as_\Phi(2)$ which are associative,
that is to say such that $m\circ_1 m=m\circ_2 m$.

\begin{prop}\label{propassociative}
Let $(A,\Phi)$ be an $\ell$EAS. The associative products in $\sym\as_\Phi(2)$ are the elements of the form
\begin{align*}
m&=*_a&&\mbox{ or }&m&=*_a^{op},
\end{align*}
where $a\in A$ is such that $\Phi(a\otimes a)=a\otimes a$.
The products $m\in \sym\as_\Phi(2)$ such that $m\circ_2 m=0$ are the elements of the form
\[m=*_a,\]
where $a\in A$ is such that $\Phi(a\otimes a)=0$.
\end{prop}

\begin{proof}
Let $m=*_a+*_b^{op}\in \sym\as_\Phi(2)$. Let $V=T_A(\mathrm{Vect}(x,y,z))$ be the free $\Phi$-associative algebra 
generated by three elements $x,y,z$. In $V$:
\begin{align*}
m\circ (\id \otimes m)(x\otimes y\otimes z)&=m(x\otimes (ayz+bzy))\\
&=\tau\circ \Phi(a\otimes a)xyz+ab yzx+\tau\circ \Phi(a\otimes b)xzy+bbzyx,\\
m\circ(m\otimes \id)(x\otimes y\otimes z)&=m((axy+byx)\otimes z)\\
&=aaxyz+bayxz+\tau\circ \Phi(b\otimes a)zxy+\tau\circ \Phi(b\otimes b)zxy.
\end{align*}
\begin{enumerate}
\item If $m$ is associative, identifying the terms in $yzx$, we find $a\otimes b=0$, so $a=0$ or $b=0$. 
Identifying the terms in $xyz$, we find that $\tau\circ \Phi(a\times a)=a\otimes a$. 
Similarly, the identification of the terms in $zyx$ gives that $\Phi(b\otimes b)=b\otimes b$.
Conversely, if $\Phi(a)=a\otimes a$, then
\begin{align*}
a\circ_1 a&=aa,& a\circ_2a&=\Phi(a\otimes a)=aa,
\end{align*}
so $*_a$ is associative, and its opposite $*_a^{op}$ is associative too.\\

\item If $m\circ_2 m=0$, identifying the term in $zyx$, we find that $b=0$. Identifying the term in $xyz$,
we find that $\tau\circ \Phi(a\times a)=0$. Conversely, if $\Phi(a\otimes a)=0$, then
\begin{align*}
a\circ_2a&=\Phi(a\otimes a)=0.\qedhere
\end{align*}\end{enumerate} \end{proof}

\begin{remark}
If $(A,\Phi)$ is the linearization of an EAS $(\Omega,\rightarrow,\triangleright)$, we obtain that:
\begin{itemize}
\item The associative elements $m\in \sym\as_\Phi(2)$ are the elements of 
the form
\begin{align*}
m&=\sum_{\alpha\in \Omega} \lambda_\alpha *_\alpha&\mbox{or}&&
m&=\sum_{\alpha\in \Omega} \lambda_\alpha *_\alpha^{op},
\end{align*}
such that
\begin{align}
\label{EQ11} &\forall (\alpha,\beta)\in \Omega^2,&
\lambda_\alpha \lambda_\beta=\sum_{\substack{(\gamma,\delta)\in \Omega^2,\\ \phi(\gamma,\delta)
=(\alpha,\beta)}}\lambda_\gamma \lambda_\delta.
\end{align}
\item The elements $m\in \sym\as_\Phi(2)$ such that $m\circ_2 m=0$ are the elements of the form
\begin{align*}
m&=\sum_{\alpha\in \Omega} \lambda_\alpha *_\alpha,
\end{align*}
such that
\begin{align}
\label{EQ12} &\forall (\alpha,\beta)\in \Omega^2,&
\sum_{\substack{(\gamma,\delta)\in \Omega^2,\\ \phi(\gamma,\delta)=(\alpha,\beta)}}\lambda_\gamma \lambda_\delta=0.
\end{align}\end{itemize}\end{remark}

\begin{example}
Working with $\eas(\Omega)$, then $\phi(\alpha,\beta)=(\beta,\alpha)$ and Condition (\ref{EQ11}) is empty:
any linear combination of $*_\alpha$ is associative, as well as their opposite. 
\end{example}

\begin{example} Let us give the associative products for EAS of cardinality two.
We only mention the spans of $*_\alpha$, their opposite should be added. Here, $\lambda,\mu$ are scalars.
\[\begin{array}{|c|c|c|}
\hline\mbox{Cases}&\mbox{Associative products}&m\circ_2 m=0\\
\hline \hline \mathbf{A1}&\lambda *_a&\lambda (*_a-*_b)\\
\hline \mathbf{A2}&\lambda *_a&\lambda (*_a-*_b)\\
\hline \mathbf{C1}&\lambda *_a&0\\
\hline \mathbf{C3}&\lambda *_a,\:\lambda *_b&0\\
\hline \mathbf{C5}&\lambda *_b&0\\
\hline \mathbf{C6}&\lambda *_a&0\\
\hline \mathbf{E1'}-\mathbf{E2'}&\lambda *_a&\lambda (*_a-*_b)\\
\hline \mathbf{E3'}&\lambda *_a\:\lambda *_b&\lambda (*_a-*_b)\\
\hline \mathbf{F1}&\lambda *_a&\lambda (*_a-*_b)\\
\hline \mathbf{F3}&\lambda *_a+\mu *_b&0\\
\hline \mathbf{F4}&\lambda (*_a+*_b),\: \lambda *_a&0\\
\hline \mathbf{H1}&\lambda *_a&0\\
\hline \mathbf{H2}&\lambda (*_a+*_b),\: \lambda *_a&0\\
\hline \end{array}\]
\end{example}

\begin{cor}
Let $(\Omega,\star)$ be a group. The nonzero associative products in $\sym\as_{\eas(\Omega,\star)}$
or in  $\sym\as_{\eas'(\Omega,\star)}$ are the elements of one of the form
\begin{align*}
&\lambda \sum_{\alpha \in H} *_\alpha&&\mbox{or}&\lambda \sum_{\alpha \in H} *_\alpha^{op},
\end{align*}
where $\lambda$ is a nonzero scalar and $H$ is a subgroup of $\Omega$.
\end{cor}

\begin{proof}
\textit{Case of $\eas(\Omega,\star)$}. Then (\ref{EQ11}) becomes:
\begin{align*}
&\forall (\alpha,\beta)\in \Omega^2,&\lambda_{\alpha\star\beta} \lambda_\alpha=\lambda_\alpha \lambda_\beta.
\end{align*}
Let $H=\{\alpha \in \Omega, \lambda_\alpha\neq 0\}$. We assume that $H$ is nonempty. If $\alpha \in H$,
for any $\beta \in H$, $\lambda_{\alpha \star \beta}=\lambda_\beta$. In particular:
\begin{itemize}
\item If $\beta \in H$, then $\alpha \star \beta \in H$.
\item If $\beta=e_\Omega$, then $\lambda_\alpha=\lambda_{e_\Omega}\neq 0$: $e_\Omega \in H$.
\item If $\beta=\alpha^{\star -1}$, $\lambda_{e_\Omega}=\lambda_{\alpha^{\star-1}}\neq 0$: $\alpha^{\star-1}\in H$.
\end{itemize}
Therefore, $H$ is a subgroup of $\Omega$. Let $\alpha,\beta \in I$, then $\alpha'=\alpha\star \beta^{-1}\in H$.
From (\ref{EQ11}), we deduce that $\lambda_{\alpha'\star \beta}=\lambda_\alpha=\lambda_\beta$. Let $\lambda$
be the common value of $\lambda_\alpha$ for any $\alpha \in H$; the result is the immediate.\\

\textit{Case of $\eas'(\Omega,\star)$}. Then (\ref{EQ11}) becomes:
\begin{align*}
&\forall (\alpha,\beta)\in \Omega^2,&\lambda_{\alpha\star\beta^{\star-1}} \lambda_\alpha=\lambda_\alpha \lambda_\beta.
\end{align*}
The proof is similar to the case of $\eas(\Omega,\star)$. \end{proof}

\subsection{Operadic morphisms between $\opas$ and $\as_\Omega$}

\begin{prop}\label{prop3.14}
Let $(\Omega,\rightarrow,\triangleright)$ and let $(A,\Phi)$ be its linearization, that is to say $A=\K\Omega$ and:
\[\Phi:\left\{\begin{array}{rcl}
A\otimes A&\longrightarrow&A\otimes A\\
\alpha\otimes \beta&\longrightarrow&\alpha \rightarrow \beta \otimes \alpha \triangleright \beta,
\end{array}\right.\]
where $\alpha,\beta\in \Omega$.
The following defines an operad morphism:
\begin{align*}
\Theta_\Omega&:\left\{\begin{array}{rcl}
\opas&\longrightarrow&\as_\Phi\\
*_{\alpha,\beta}&\longrightarrow&*_{\alpha \triangleright \beta}.
\end{array}\right.
\end{align*}
\end{prop}

\begin{proof}
Let us consider an $\Omega$-associative algebra $(A,(*_\alpha)_{\alpha \in \Omega})$.
For any $(\alpha,\beta)\in \Omega^2$, we put $*_{\alpha,\beta}=*_{\alpha \triangleright \beta}$.
Then, for any $x,y,z\in A$:
\begin{align*}
x*_{\alpha,\beta \rightarrow \gamma}(y *_{\beta,\gamma} z)
&=x*_{\alpha \triangleright(\beta \rightarrow \gamma)} (y *_{\beta \rightarrow \gamma} z)\\
&=(x*_{(\alpha \triangleright(\beta \rightarrow \gamma))\triangleright (\beta \rightarrow \gamma)} y)
*_{(\alpha \triangleright(\beta \rightarrow \gamma))\rightarrow (\beta \rightarrow \gamma)} z\\
&=(x*_{(\alpha\rightarrow \beta) \triangleright\gamma} y)
*_{\alpha \triangleright\beta} z\\
&=(x*_{\alpha\rightarrow \beta,\gamma}y)*_{\alpha,\beta} z.
\end{align*} 
Hence, $(A,(*_{\alpha,\beta})_{\alpha,\beta,\in \Omega})$ is a 2-parameter $\Omega$-associative algebra,
which implies the existence of the operadic morphism $\Theta_\Omega$. 
\end{proof}

\begin{prop}\label{prop3.15}
Let $(\Omega,\rightarrow)$ be an associative semigroup with the right inverse condition. 
We consider the EAS $\Omega'=\eas(\Omega,\rightarrow)\times \eas'(\Omega,\rightarrow)$
and denote by $(A',\Phi')$ its linearization. The following defines a surjective operad morphism:
\begin{align*}
\Theta'_\Omega&:\left\{\begin{array}{rcl}
\opas&\longrightarrow&\as_{\Phi'}\\
*_{\alpha,\beta}&\longrightarrow&*_{(\alpha,\beta)}.
\end{array}\right.
\end{align*}\end{prop}

\begin{proof}
The EAS structure of $\Omega'$ is given by:
\begin{align*}
&\forall (\alpha,\beta,\gamma,\delta)\in \Omega^4,&(\alpha,\beta)\rightarrow (\gamma,\delta)
&=(\alpha \rightarrow \gamma,\delta),\\
&&(\alpha,\beta)\triangleright (\gamma,\delta)&=(\alpha,\beta\triangleright \delta).
\end{align*}
Let $(A,(*_{(\alpha,\beta)})_{(\alpha,\beta) \in \Omega^2})$ be an $\Omega'$-associative algebra.
For any $(\alpha,\beta)\in \Omega'$, we put $*_{\alpha,\beta}=*_{\alpha,\beta}$. 
Then, for any $x,y,z\in A$, using the right inverse property for the second equality:
\begin{align*}
(x*_{\alpha,\beta} y)*_{\alpha\rightarrow \beta,\gamma} z
&=(x*_{(\alpha,\beta)}y)*_{(\alpha\rightarrow \beta,\gamma)}z\\
&=(x*_{(\alpha,(\beta\rightarrow \gamma) \triangleright \gamma)}y)*_{(\alpha\rightarrow \beta,\gamma)}z\\
&=(x*_{(\alpha,\beta \rightarrow \gamma) \triangleright (\beta,\gamma)} y)
*_{(\alpha,\beta \rightarrow \gamma) \rightarrow (\beta,\gamma)} z\\
&=x*_{(\alpha,\beta\rightarrow \gamma)}(y*_{(\beta,\gamma)}z)\\
&=x*_{\alpha,\beta\rightarrow \gamma}(y*_{\beta,\gamma}z).
\end{align*}
Hence, $(A,(*_{\alpha,\beta})_{\alpha,\beta \in \Omega})$ is a 2-parameter $\Omega$-associative algebra. 
This implies the existence of the morphism $\Theta'_\Omega$. \end{proof}

\begin{remark}
Except if $\omega=|\Omega|=1$, this morphism is not bijective: the dimension of $\opas(3)$ is $(2\omega-1)\omega^3$, 
whereas the dimension of $\as_{\Phi'}(3)$ is $\omega^4$. 
\end{remark}

\section{Links with other operads}

\subsection{Post-Lie and ComTriAs algebras}

Let us consider post-Lie algebras \cite{Vallette2007}, see also \cite{postLie1,postLie2,postLie3,postLie4,postLie5,postLie6,postLie7} 
for applications and developments. Recall that a post-Lie algebra is a family $(A,*,\{,\})$ where  $A$ is a vector space
 and $*$ and $\star$ are bilinear products on $A$ such that $(A,\{,\})$ is a Lie algebra and,  for any $x,y,z\in A$:
\begin{align*}
x*\{y,z\}&=(x*y)*z - x*(y*z) - (x*z)*y + x*(z*y),\\
\{x,y\}*z&=\{x*z,y\}+\{x,,y*z\}.
\end{align*}
Let us start with the Koszul dual of the operad of post-Lie algebras, namely the operad of ComTriAs algebras \cite{LodayEncyclopedia}:

\begin{defi}
A ComTriAs algebra is a family $(A,\cdot,\star)$, where $A$ is a vector space and $\cdot$ and $\star$ are bilinear products
on $A$ such that for any $x,y,z\in A$:
\begin{align*}
x\cdot y&=y\cdot x,\\
(x\cdot y)\cdot z&=x\cdot (y\cdot z),\\
(x\star y)\star z&=x\star(y\star z),\\
(x\star y)\star z&=x\star(y\cdot z),\\
(x\cdot y)\star z&=x(y\star z).
\end{align*}
Note that the products $\cdot$ and $*$ are respectively denoted by $\perp$ and $\dashv$ in \cite{LodayEncyclopedia}.\\
\end{defi}

Apart from the first one, these axioms are the ones of a particular example of generalized associative algebra:

\begin{prop}
Let $\Omega$ be the EAS $\mathbf{C3}$ (that is to say the EAS associated to the semigroup $(\Z/2\Z,\times)$.
Then any ComTriAs algebra $(A,\cdot,*)$ is an $\Omega$-associative algebra, with $*_{\overline{0}}=\star$
and $*_{\overline{1}}=\cdot$.
\end{prop}

Consequently, we obtain an operad morphism from the operad of $\Omega$-associative algebra to the operad
of ComTriAs algebras. Using Koszul duality:

\begin{cor}
Let $(V,\Phi)$ be the $\ell$EAS dual to $\mathbf{C3}$: $V$ is two-dimensional, with basis $(e_1,e_2)$,
and the basis of $\Phi$ is the basis $(e_1\otimes e_1,e_1\otimes e_2,e_2\otimes e_1,e_2\otimes e_2)$ is
\[\begin{pmatrix}
1&0&0&0\\
1&0&0&0\\
0&1&0&0\\
0&0&0&1
\end{pmatrix}.\]
Then any opposite $\Phi$-associative algebra is a post-Lie algebra,  with for any $x,y\in A$:
\begin{align*}
\{x,y\}&=x *_2 y-y*_2 x,&x*y&=x*_1 y.
\end{align*}
\end{cor}

We conjecture that the associated operad morphism from the operad of post-Lie algebras into the operad of opposite 
$\Phi$-associative algebra is injective.

\subsection{Diassociative and dendriform algebras}

\begin{defi}\cite{Loday2001}
A diassociative algebra is a family $(A,\dashv,\vdash)$ where $A$ is a vector space and $\dashv$ and $\vdash$ are bilinear products
on $A$ such that for any $x,y,z\in A$:
\begin{align}
\label{eq2as1} (x\dashv y)\dashv z&=x\dashv (y\dashv z),\\
\label{eq2as1bis} (x\dashv y)\dashv z&=x\dashv (y\vdash z),\\
\label{eq2as2} (x\vdash y)\dashv z&=x\vdash (y\dashv z),\\
\label{eq2as3} (x\dashv y)\vdash z&=x\vdash (y\vdash z),\\
\label{eq2as4} (x\vdash y)\vdash z&=x\vdash (y\vdash z).
\end{align}
\end{defi}

\begin{prop}
Let $(A,\dashv,\vdash)$ be a diassociative algebra.
\begin{enumerate}
\item $(A,\dashv,\vdash)$ is an opposite $\Omega$-associative algebra, with the EAS laws:
\begin{align*}
&\begin{array}{|c||c|c|c|c|}
\hline \rightarrow&\dashv&\vdash\\
\hline\hline \dashv&\dashv&\vdash\\
\hline \vdash&\vdash&\vdash\\
\hline
\end{array}
&&\begin{array}{|c||c|c|c|c|}
\hline \triangleright&\dashv&\vdash\\
\hline\hline \dashv&\dashv&\dashv\\
\hline \vdash&\vdash&\vdash\\
\hline
\end{array}\:\mbox{ or }\:
\begin{array}{|c||c|c|c|c|}
\hline \triangleright&\dashv&\vdash\\
\hline\hline \dashv&\vdash&\dashv\\
\hline \vdash&\vdash&\vdash\\
\hline
\end{array}
\end{align*}
These EAS are  isomorphic to $\mathbf{C3}$ and $\mathbf{C6}$. 
\item $(A,\dashv,\vdash)$ is an $\Omega$-associative algebra, with the EAS laws:
\begin{align*}
&\begin{array}{|c||c|c|c|c|}
\hline \rightarrow&\dashv&\vdash\\
\hline\hline \dashv&\dashv&\dashv\\
\hline \vdash&\dashv&\vdash\\
\hline
\end{array}
&&\begin{array}{|c||c|c|c|c|}
\hline \triangleright&\dashv&\vdash\\
\hline\hline \dashv&\dashv&\dashv\\
\hline \vdash&\vdash&\dashv\\
\hline
\end{array}\:\mbox{ or }\:\begin{array}{|c||c|c|c|c|}
\hline \triangleright&\dashv&\vdash\\
\hline\hline \dashv&\dashv&\dashv\\
\hline \vdash&\vdash&\vdash\\
\hline
\end{array}
\end{align*}
These EAS are  isomorphic to $\mathbf{C6}$ and $\mathbf{C3}$. 
\end{enumerate}
\end{prop}

\begin{proof}
\begin{enumerate}
\item This is a reformulation of axioms (\ref{eq2as1}), (\ref{eq2as2}), (\ref{eq2as3}) and (\ref{eq2as4}),
and of axioms (\ref{eq2as1bis}), (\ref{eq2as2}), (\ref{eq2as3}) and (\ref{eq2as4}).
\item This is a reformulation of axioms (\ref{eq2as1}), (\ref{eq2as1bis}), (\ref{eq2as2}) and (\ref{eq2as3}),
and of axioms (\ref{eq2as1}), (\ref{eq2as1bis}), (\ref{eq2as2}) and (\ref{eq2as4}). \qedhere
\end{enumerate}
\end{proof}

 Using Koszul duality, we obtain dendriform algebras: recall that a dendriform algebra is a family $(A,\prec,\succ)$ 
 where $A$ is a vector space and $\prec$ and $\succ$ are bilinear products on $A$ such that for any $x,y,z\in A$:
\begin{align*}
(x\prec y)\prec z&=x\prec(y\prec z+y\succ z),\\
(x\succ y)\prec z&=x\succ (y\prec z),\\
x\succ (y\succ z)&=(x\prec y+x\succ y)\succ z.
\end{align*}

\begin{cor}
\begin{enumerate}
\item Let $(V,\Phi)$ be one of the two following 2-dimensional $\ell$-EAS, 
where the matrix of $\Phi$ is expressed in the basis $(e_1\otimes e_1,e_1\otimes e_2,e_2\otimes e_1,e_2\otimes e_2)$:
\begin{align*}
&\begin{pmatrix}
1&0&0&0\\
0&0&0&0\\
0&1&0&0\\
0&0&1&1
\end{pmatrix},&&
\begin{pmatrix}
0&0&0&0\\
1&0&0&0\\
0&1&0&0\\
0&0&1&1
\end{pmatrix}.
\end{align*}
Then  any  $\Phi$-associative algebra is a dendriform algebra, with $\prec=*_1$ and $\succ=*_2$. 
\item Let $(V,\Phi)$ is one of the two following 2-dimensional $\ell$-EAS, 
where the matrix of $\Phi$ is expressed in the basis $(e_1\otimes e_1,e_1\otimes e_2,e_2\otimes e_1,e_2\otimes e_2)$:
\begin{align*}
&\begin{pmatrix}
1&0&0&0\\
1&0&0&0\\
0&1&0&0\\
0&0&1&0
\end{pmatrix},&&
\begin{pmatrix}
1&0&0&0\\
1&0&0&0\\
0&1&0&0\\
0&0&0&1
\end{pmatrix}.
\end{align*}
Then  any  opposite $\Phi$-associative algebra is a dendriform algebra, with $\prec=*_1$ and $\succ=*_2$. 
\end{enumerate}
\end{cor}

\subsection{Triassociative and tridendriform algebras}

\begin{defi} \cite{Loday2004}
A triassociative algebra is a family $(A,\dashv,\vdash,\perp)$ where $A$ is a vector space and $\dashv$, $\vdash$ and $\perp$ are bilinear products on $A$ such that for any $x,y,z\in A$:
\begin{align}
\label{eq3as1} (x\dashv y)\dashv z&=x\dashv (y\dashv z),\\
\label{eq3as1bis} (x\dashv y)\dashv z&=x\dashv (y\vdash z),\\
\label{eq3as1ter}(x\dashv y)\dashv z&=x\dashv (y\perp z),\\
\label{eq3as2} (x\vdash y)\dashv z&=x\vdash (y\dashv z),\\
\label{eq3as3} (x\perp y)\dashv z&=x\perp (y\dashv z),\\
\label{eq3as4} (x\dashv y)\perp z&=x\perp (y\vdash z),\\ 
\label{eq3as5} (x\vdash y)\perp z&=x\vdash (y\perp z),\\
\label{eq3as6} (x\dashv y)\vdash z&=x\vdash (y\vdash z),\\
\label{eq3as6bis} (x\perp y)\vdash z&=x\vdash (y\vdash z),\\
\label{eq3as6ter} (x\vdash y)\vdash z&=x\vdash (y\vdash z).
\end{align}
\end{defi}

\begin{prop}
Let $(A,\dashv,\vdash,\perp)$ be a triassociative algebra.
\begin{enumerate}
\item $(A,\dashv,\vdash,\perp)$ is an opposite $\Omega$-associative algebra, with the EAS laws:
\begin{align*}
&\begin{array}{|c||c|c|c|c|c|}
\hline \rightarrow&\dashv&\vdash&\perp\\
\hline\hline \dashv&\dashv&\vdash&\perp\\
\hline \vdash&\vdash&\vdash&\vdash\\
\hline \perp&\perp&\vdash&\perp\\
\hline \end{array}
&&\begin{array}{|c||c|c|c|c|c|}
\hline \triangleright&\dashv&\vdash&\perp\\
\hline\hline \dashv&\dashv&\dashv&\dashv\\
\hline \vdash&\vdash&\vdash&\vdash\\
\hline \perp&\vdash&\perp&\perp\\
\hline \end{array}\:\mbox{ or }\:
\begin{array}{|c||c|c|c|c|c|}
\hline \triangleright&\dashv&\vdash&\perp\\
\hline\hline \dashv&\vdash&\dashv&\dashv\\
\hline \vdash&\vdash&\vdash&\vdash\\
\hline \perp&\vdash&\perp&\perp\\
\hline \end{array}\:\mbox{ or }\:
\begin{array}{|c||c|c|c|c|c|}
\hline \triangleright&\dashv&\vdash&\perp\\
\hline\hline \dashv&\perp&\dashv&\dashv\\
\hline \vdash&\vdash&\vdash&\vdash\\
\hline \perp&\vdash&\perp&\perp\\
\hline \end{array}
\end{align*}
\item $(A,\dashv,\vdash,\perp)$ is an $\Omega$-associative algebra, with the EAS laws:
\begin{align*}
&\begin{array}{|c||c|c|c|c|c|}
\hline \rightarrow&\vdash&\dashv&\perp\\
\hline\hline \vdash&\vdash&\dashv&\perp\\
\hline \dashv&\dashv&\dashv&\dashv\\
\hline \perp&\perp&\dashv&\perp\\
\hline \end{array}
&&\begin{array}{|c||c|c|c|c|c|}
\hline \triangleright&\vdash&\dashv&\perp\\
\hline\hline \vdash&\vdash&\vdash&\vdash\\
\hline \dashv&\dashv&\dashv&\dashv\\
\hline \perp&\dashv&\perp&\perp\\
\hline \end{array}\:\mbox{ or }\:
\begin{array}{|c||c|c|c|c|c|}
\hline \triangleright&\vdash&\dashv&\perp\\
\hline\hline \vdash&\dashv&\vdash&\vdash\\
\hline \dashv&\dashv&\dashv&\dashv\\
\hline \perp&\dashv&\perp&\perp\\
\hline \end{array}\:\mbox{ or }\:
\begin{array}{|c||c|c|c|c|c|}
\hline \triangleright&\vdash&\dashv&\perp\\
\hline\hline \vdash&\perp&\vdash&\vdash\\
\hline \dashv&\dashv&\dashv&\dashv\\
\hline \perp&\dashv&\perp&\perp\\
\hline \end{array}
\end{align*}
\end{enumerate}
\end{prop}

\begin{proof}
\begin{enumerate}
\item This is a reformulation of axioms((\ref{eq3as1})  or (\ref{eq3as1bis}) or  (\ref{eq3as1ter}))
and (\ref{eq3as2}) -- (\ref{eq3as6ter}).  
 \item This is a reformulation of axioms (\ref{eq3as1}) -- (\ref{eq3as5}),
and ((\ref{eq3as6}) or   (\ref{eq3as6bis}) or  (\ref{eq3as6ter})).  \qedhere
\end{enumerate}
\end{proof}

Using Koszul duality, we obtain tridendriform algebras \cite{Loday2001,Chapoton2000,Novelli2006},
that is to say families $(A,\prec,\succ,\cdot)$ where $A$ is a vector space and $\prec$, $\succ$ and $\cdot$ 
are bilinear products on $A$ such that for any $x,y,z\in A$:
\begin{align*}
(x\prec y)\prec z&=x\prec (y\prec z+y\succ z+y\cdot z),\\
(x\succ y)\prec z&=x\succ (y\prec z),\\
 x\succ (y\succ z)&=(x\prec  y+x\succ y+x\cdot y)\succ z,\\
(x\succ y)\cdot z&=x\succ (y\cdot z),\\
(x\prec y)\cdot z&=x\cdot (y\succ z),\\
(x\cdot y)\prec z&=x\cdot (y\prec z),\\
(x\cdot y)\cdot z&=x\cdot (y\cdot z).
\end{align*}

\begin{cor}
Let $(V,\Phi)$ is one of the three following 3-dimensional $\ell$-EAS, 
where the matrix of $\Phi$ is expressed in the basis $(e_1\otimes e_1,e_1\otimes e_2,e_1\otimes e_3, e_2\otimes e_1,e_2\otimes e_2,
e_2\otimes e_3,e_3\otimes e_1,e_3\otimes e_2,e_3\otimes e_3)$:
\begin{align*}
&\begin{pmatrix}
1&0&0&0&0&0&0&0&0\\
0&0&0&0&0&0&1&0&0\\
0&0&0&0&1&0&0&0&0\\
0&0&0&0&1&0&0&0&0\\
0&0&0&0&1&0&0&0&0\\
0&0&0&0&0&0&0&1&0\\
0&0&0&0&0&1&0&0&0\\
0&0&0&0&0&0&0&0&1
\end{pmatrix},&&
\begin{pmatrix}
0&1&0&0&0&0&0&0&0\\
0&0&0&0&0&0&1&0&0\\
0&0&0&0&1&0&0&0&0\\
0&0&0&0&1&0&0&0&0\\
0&0&0&0&1&0&0&0&0\\
0&0&0&0&0&0&0&1&0\\
0&0&0&0&0&1&0&0&0\\
0&0&0&0&0&0&0&0&1
\end{pmatrix},&&
\begin{pmatrix}
0&0&1&0&0&0&0&0&0\\
0&0&0&0&0&0&1&0&0\\
0&0&0&0&1&0&0&0&0\\
0&0&0&0&1&0&0&0&0\\
0&0&0&0&1&0&0&0&0\\
0&0&0&0&0&0&0&1&0\\
0&0&0&0&0&1&0&0&0\\
0&0&0&0&0&0&0&0&1
\end{pmatrix}.
\end{align*}
Then  any  $\Phi$-associative algebra is a tridendriform algebra, with $\prec=*_1$, $\succ=*_2$ and $\perp=*_3$. 
Any opposite $\Phi$-associative algebra is a tridendriform algebra, with $\prec=*_2$, $\succ=*_1$ and $\perp=*_3$. 
\end{cor}

\subsection{Dual duplicial and duplicial algebras}

\begin{defi}\cite{LodayEncyclopedia}
A dual duplicial algebra is a family $(A,\prec,\succ)$ where $A$ is a vector space and $\prec$ and $\succ$ are bilinear products
on $A$ such that for any $x,y,z\in A$:
\begin{align}
\label{eqdup1}(x\prec y)\prec z&=x\prec (y\prec z),\\
\label{eqdup2} (x\prec x)\succ z&=0,\\
\label{eqdup3} (x\succ y)\prec z&=x\succ (y\prec z),\\
\label{eqdup4} 0&=x\prec (y\succ z),\\
\label{eqdup5} (x\succ y)\succ z&=x\succ (y\succ z).
\end{align}
\end{defi}

\begin{prop}
Let $(V,\Phi)$be the following 2-dimensional $\ell$-EAS, 
where the matrix of $\Phi$ is expressed in the basis $(e_1\otimes e_1,e_1\otimes e_2,e_2\otimes e_1,e_2\otimes e_2)$:
\begin{align*}
&\begin{pmatrix}
1&0&0&0\\
0&0&0&0\\
0&1&0&0\\
0&0&0&1
\end{pmatrix}.
\end{align*}
Then any dual duplicial algebra $(A,\prec,\succ)$ is an opposite $\Phi$-associative algebra, with $*_1=\prec$ and $*_2=\succ$,
and a$\Phi$-associative algebra, with $*_1=\succ$ and $*_2=\prec$.
\end{prop}

\begin{proof}
This is a reformulation of axioms (\ref{eqdup1}) --  (\ref{eqdup3}) and (\ref{eqdup5}),
and of axioms  (\ref{eqdup1}) and (\ref{eqdup3})  -- (\ref{eqdup5}).
\end{proof}

Using Koszul duality, we recover duplicial algebra \cite{Loday2008}, that is to say families $(A,\prec,\succ)$ 
 where $A$ is a vector space and $\prec$ and $\succ$ are bilinear products on $A$ such that for any $x,y,z\in A$:
\begin{align*}
(x\prec y)\prec z&=x\prec(y\prec z),\\
(x\succ y)\prec z&=x\succ (y\prec z),\\
x\succ (y\succ z)&=(x\succ y)\succ z.
\end{align*}

\begin{cor}
Let $(V,\Phi)$be the following 2-dimensional $\ell$-EAS, 
where the matrix of $\Phi$ is expressed in the basis $(e_1\otimes e_1,e_1\otimes e_2,e_2\otimes e_1,e_2\otimes e_2)$:
\begin{align*}
&\begin{pmatrix}
1&0&0&0\\
0&0&1&0\\
0&0&0&0\\
0&0&0&1
\end{pmatrix}.
\end{align*}
Then any $\Phi$-associative algebra is a duplicial algebra, with  $*_1=\prec$ and $*_2=\succ$,
and any opposite $\Phi$-associative algebra is a duplicial algebra, with $*_1=\succ$ and $*_2=\prec$.
\end{cor}

\bibliographystyle{amsplain}
\addcontentsline{toc}{section}{References}
\bibliography{biblio}

\end{document}

%% file: arbresbinaires.tex
\newcommand{\bdtroisun}{
\begin{tikzpicture}[line cap=round,line join=round,>=triangle 45,x=0.2cm,y=0.2cm]
\clip(-3,0.) rectangle (1.2,4.);
\draw [line width=.5pt] (0.,0.)-- (0.,1.);
\draw [line width=.5pt] (0.,1.)-- (-1.,2.);
\draw [line width=.5pt] (0.,1.)-- (1.,2.);
\draw [line width=.5pt] (-1.,2.)-- (0.,3.);
\draw [line width=.5pt] (-1.,2.)-- (-2.,3.);
\draw(-0.7,1.) node {\tiny 1};
\draw(-1.5,1.7) node {\tiny 2};
\end{tikzpicture}}
\newcommand{\bdtroisdeux}{
\begin{tikzpicture}[line cap=round,line join=round,>=triangle 45,x=0.2cm,y=0.2cm]
\clip(-2.4,0.) rectangle (2.2,4.);
\draw [line width=.5pt] (0.,0.)-- (0.,1.);
\draw [line width=.5pt] (0.,1.)-- (-1.,2.);
\draw [line width=.5pt] (0.,1.)-- (1.,2.);
\draw [line width=.5pt] (1.,2.)-- (2.,3.);
\draw [line width=.5pt] (1.,2.)-- (0.,3.);
\draw(0.7,1.) node {\tiny 1};
\draw(1.5,1.7) node {\tiny 2};
\end{tikzpicture}}
\newcommand{\bdquatreun}{
\begin{tikzpicture}[line cap=round,line join=round,>=triangle 45,x=0.2cm,y=0.2cm]
\clip(-3.2,0.) rectangle (1.2,4.);
\draw [line width=.5pt] (0.,0.)-- (0.,1.);
\draw [line width=.5pt] (0.,1.)-- (-1.,2.);
\draw [line width=.5pt] (0.,1.)-- (1.,2.);
\draw [line width=.5pt] (-1.,2.)-- (0.,3.);
\draw [line width=.5pt] (-1.,2.)-- (-2.,3.);
\draw [line width=.5pt] (-3.,4.)-- (-2.,3.);
\draw [line width=.5pt] (-2.,3.)-- (-1.,4.);
\draw(-0.7,1.) node {\tiny 1};
\draw(-1.5,1.7) node {\tiny 2};
\draw(-2.3,2.4) node {\tiny 3};
\end{tikzpicture}}
\newcommand{\bdquatredeux}{
\begin{tikzpicture}[line cap=round,line join=round,>=triangle 45,x=0.2cm,y=0.2cm]
\clip(-2.2,0.) rectangle (1.2,4.);
\draw [line width=.5pt] (0.,0.)-- (0.,1.);
\draw [line width=.5pt] (0.,1.)-- (-1.,2.);
\draw [line width=.5pt] (0.,1.)-- (1.,2.);
\draw [line width=.5pt] (-1.,2.)-- (-2.,3.);
\draw [line width=.5pt] (-1.,2.)-- (0.,3.);
\draw [line width=.5pt] (-1.,4.)-- (0.,3.);
\draw [line width=.5pt] (0.,3.)-- (1.,4.);
\draw(-0.7,1.) node {\tiny 1};
\draw(-1.5,1.7) node {\tiny 2};
\draw(0.1,2.4) node {\tiny 3};
\end{tikzpicture}}
\newcommand{\bdquatretrois}{
\begin{tikzpicture}[line cap=round,line join=round,>=triangle 45,x=0.2cm,y=0.2cm]
\clip(-1.2,0.) rectangle (2.2,4.);
\draw [line width=0.4pt] (0.,0.)-- (0.,1.);
\draw [line width=0.4pt] (0.,1.)-- (-1.,2.);
\draw [line width=0.4pt] (0.,1.)-- (1.,2.);
\draw [line width=0.4pt] (1.,2.)-- (2.,3.);
\draw [line width=0.4pt] (1.,2.)-- (0.,3.);
\draw [line width=0.4pt] (-1.,4.)-- (0.,3.);
\draw [line width=0.4pt] (0.,3.)-- (1.,4.);
\draw(0.7,1.) node {\tiny 1};
\draw(1.5,1.7) node {\tiny 2};
\draw(-0.1,2.4) node {\tiny 3};
\end{tikzpicture}}
\newcommand{\bdquatrequatre}{
\begin{tikzpicture}[line cap=round,line join=round,>=triangle 45,x=0.2cm,y=0.2cm]
\clip(-1.2,0.) rectangle (3.2,4.);
\draw [line width=.5pt] (0.,0.)-- (0.,1.);
\draw [line width=.5pt] (0.,1.)-- (-1.,2.);
\draw [line width=.5pt] (0.,1.)-- (1.,2.);
\draw [line width=.5pt] (1.,2.)-- (0.,3.);
\draw [line width=.5pt] (1.,2.)-- (2.,3.);
\draw [line width=.5pt] (1.,4.)-- (2.,3.);
\draw [line width=.5pt] (2.,3.)-- (3.,4.);
\draw(0.7,1.) node {\tiny 1};
\draw(1.5,1.7) node {\tiny 2};
\draw(2.3,2.4) node {\tiny 3};
\end{tikzpicture}}
\newcommand{\bdquatrecinq}{
\begin{tikzpicture}[line cap=round,line join=round,>=triangle 45,x=0.2cm,y=0.2cm]
\clip(-2.2,0.) rectangle (2.2,3.);
\draw [line width=.5pt] (0.,0.)-- (0.,1.);
\draw [line width=.5pt] (0.,1.)-- (-2.,3.);
\draw [line width=.5pt] (0.,1.)-- (1.,2.);
\draw [line width=.5pt] (1.,2.)-- (2.,3.);
\draw [line width=.5pt] (1.,2.)-- (0.5,3.);
\draw [line width=.5pt] (-0.5,3.)-- (-1.,2.);
\draw(-0.7,1.) node {\tiny 1};
\draw(-1.5,1.7) node {\tiny 2};
\draw(1.5,1.7) node {\tiny 3};
\end{tikzpicture}}